\documentclass[a4paper]{article}

\usepackage[utf8]{inputenc}
\usepackage{amsfonts,amssymb, amsmath, amsthm,mathtools}
\usepackage{color}
\usepackage{caption}
\usepackage{subcaption}
\usepackage{enumitem}
\usepackage{hyperref}
\usepackage{tikz}

\newcommand{\fignote}[1]{\textcolor{red}{#1}}

\newtheorem{lemma}{Lemma}[section]
\newtheorem{example}[lemma]{Example}
\newtheorem{remark}[lemma]{Remark}
\newtheorem{proposition}[lemma]{Proposition}
\newtheorem{theorem}[lemma]{Theorem}
\newtheorem{corollary}[lemma]{Corollary}
\theoremstyle{remark}
\newtheorem{assumption}[lemma]{Assumption}

\DeclareMathOperator{\linspan}{span} % vector calculus divergence
\DeclareMathOperator{\diam}{diam} % diameter
\DeclareMathOperator*{\supp}{supp}
\newcommand{\real}{\mathbb{R}}
\newcommand{\dualp}[1]{\left\langle #1 \right\rangle} %dual pairing
\DeclareMathOperator{\ran}{Range} % diameter

\newcommand{\param}{\mathcal{P}}
\newcommand{\parama}{\bar{\mathcal{P}}}
\newcommand{\paramb}{\hat{\mathcal{P}}}

\newcommand{\problems}{\textnormal{\textsc{problems}}}
\newcommand{\offfunc}{\textnormal{\textsc{offline}}}
\newcommand{\onfunc}{\textnormal{\textsc{online}}}
\newcommand{\ondata}{\textnormal{\textsc{onlinedata}}}

\newcommand{\instance}{P}
\newcommand{\instances}{\mathbb{P}}

\newcommand{\qi}{q}
\newcommand{\Qi}{Q}

\newcommand{\kset}{K}
\newcommand{\nwidth}{d}
\newcommand{\mwidth}{\delta}
\newcommand{\swidth}{\delta^\gamma}
\newcommand{\lwidth}{\delta^{\gamma,B}}
\newcommand{\entropy}{\epsilon}
\newcommand{\bentropy}{B}

\newcommand{\sm}{s_-}
\renewcommand{\sp}{s_+}
\renewcommand{\sb}{s_b}

\newcommand{\cf}{\mathcal{F}}
\newcommand{\cb}{\mathcal{B}}
\newcommand{\diml}{{\bar{d}}}
\newcommand{\rdom}{\Xi}

\newcommand{\mua}{\bar{\mu}}
\newcommand{\mub}{\hat{\mu}}

\newcommand{\bi}{V_-}
\newcommand{\bo}[1][]{V_{+#1}}

\newcommand{\flows}{\mathcal{B}}

\begin{document}

\title{Performance bounds for Reduced Order Models with Application to Parametric Transport}
\author{D. Rim\footnote{Department of Mathematics, Washington University in St. Louis, St. Louis, MO 63105, USA, \texttt{rim@wustl.edu}}, G. Welper\footnote{Department of Mathematics, University of Central Florida, Orlando, FL 32816, USA, \texttt{gerrit.welper@ucf.edu}}}
\date{}
\maketitle

\begin{abstract}

  The Kolmogorov $n$-width is an established benchmark to judge the performance of reduced basis and similar methods that produce linear reduced spaces. Although immensely successful in the elliptic regime, this width, shows unsatisfactory slow convergence rates for transport dominated problems. While this has triggered a large amount of work on nonlinear model reduction techniques, we are lacking a benchmark to evaluate their optimal performance.

  Nonlinear benchmarks like manifold/stable/Lipschitz width applied to the solution manifold are often trivial if the degrees of freedom exceed the parameter dimension and ignore desirable structure as offline/online decompositions. In this paper, we show that the same benchmarks applied to the full reduced order model pipeline from PDE to parametric quantity of interest provide non-trivial benchmarks and we prove lower bounds for transport equations.

\end{abstract}



\section{Introduction}

Reduced order modeling seeks fast and reliable approximations of a quantity of interest $\qi(\mu)$ computed from the solution of a parametric PDE
\begin{align} \label{eq:ppde}
  \qi(\mu) & = \ell_\mu(u_\mu), &
  & \text{s.t.} & 
  F_\mu(u_\mu) & = 0, 
\end{align}
depending on parameters $\mu \in \param \subset \real^D$, with differential operator $F_\mu\colon U \to V$, continuous functional $\ell_\mu\colon U \to W$ and Banach spaces $U$, $V$ and $W$.

\subsection{Kolmogorov \texorpdfstring{$n$}{n}-width}

The Kolmogorov $n$-width is one of the most used benchmarks in reduced order modeling and measures the best possible approximation from a linear subspace $U_n \subset U$ for functions in a class $K \subset U$:
\begin{equation*}
  \nwidth_n(\kset) = \inf_{\dim U_n = n} \sup_{u \in K} \inf_{\phi \in U_n} \|u - \phi\|_U.
\end{equation*}
Applied to the solution manifold $\kset = \{u_\mu \mid \mu \in \param\} \subset U$ this provides a lower bound for all reduced order models that rely on linear approximation from some reduced space $U_n$, e.g. spanned by snapshots or POD modes. It is well known that greedy methods generate reduced spaces that asymptotically match the Kolmogorov $n$-width 
\cite{
  BinevCohenDahmenDeVorePetrovaWojtaszczyk2011,
  BuffaMadayPateraPrudhommeTurinici2012,
  DeVorePetrovaWojtaszczyk2013%
}
and that for many elliptic problems the width decreases exponentially
\cite{
  CohenDeVoreSchwab2010,
  CohenDeVoreSchwab2011,
  ChkifaCohenDeVoreSchwab2013,
  CohenDeVore2015%
}. 
For transport dominated problems, on the other hand, the situation is less satisfactory with Kolmogorov $n$-width 
\cite{
OhlbergerRave2016,
Welper2017,
ArbesGreifUrban2023%
}
\begin{align*}
  \nwidth_n(\kset) & \sim n^{-1/2}, & U & = L_2(\Omega) \\
  \nwidth_n(\kset) & \sim n^{-1}, &   U & = L_1(\Omega).
\end{align*}
This places severe restrictions on reduced order models that rely on linear approximation and is the motivation for several nonlinear methods. These include reduced order models with shifts or transforms 
\cite{
  Welper2017,
  ReissSchulzeSesterhenn2015,
  CagniartMadayStamm2019,
  CagniartCrisovanMadayEtAl2017,
  NairBalajewicz2017,
  Welper2019,
  KrahSrokaReiss2019,
  Taddei2020%
}
, optimal transport
\cite{
  IolloLombardi2014,
  RimMandli2018,
  EhrlacherLombardiMulaEtAl2019,
  RimPeherstorferMandli2023%
}
, neural networks
\cite{
LeeCarlberg2018,
KimChoiWidemannZohdi2020,
KimChoiWidemannZohdi2022,
Dahmen2023%
}
and other techniques
\cite{
  PeherstorferWillcox2015,
  Peherstorfer2018,
  GeelenWrightWillcox2023,
  BarnettFarhatMaday2023%
}
. To assess their performance, we need a corresponding nonlinear benchmark, which is the main focus of this work.

\subsection{Nonlinear Width and Entropy}
\label{sec:nonlinear-width}

\paragraph{Nonlinear Benchmarks}

For linear width, we approximate functions $u$ in a class $K$ by some linear subspace $U_n$, or equivalently $n$ coefficients in $\real^n$ for a suitable basis. In order to find benchmarks for adaptive approximation methods, we relax the linearity constraint: We project $u$ to a finite dimensional space $\real^n$ with a nonlinear encoder $E_n\colon K \to \real^n$ and then reconstruct it with a corresponding nonlinear decoder $D_n\colon \real^n \to U$, yielding the approximation $u \approx D_n \circ E_n(u)$. In addition, we require stability in form of continuity conditions and then search for the best encoder/decoder pair.
\begin{enumerate}

  \item The \emph{manifold width} \cite{DeVoreHowardMicchelli1989,DeVoreKyriazisLeviatanTikhomirov1993}, requires encoder and decoder pairs to be continuous
  \begin{equation*}
    \mwidth_n(\kset) = \inf_{\substack{D_n, \, E_n \\ \text{continuous}}} \sup_{u \in K} \|u - D_n \circ E_n (u)\|_U.
  \end{equation*}
  Although one is sometimes only interested in the continuity of the decoder $D_n$ or the composition $D_n \circ E_n$, the continuity of the encoder prevents degenerate decoders like space-filling curves.

  \item A refined variant is the \emph{stable manifold width} \cite{CohenDeVorePetrovaWojtaszczyk2022}, where the encoder and decoder are $\gamma$-Lipschitz continuous for some $\gamma \ge 1$
  \begin{equation*}
    \swidth_n(\kset) = \inf_{\substack{D_n, \, E_n, Y \\ \gamma-\text{Lipschitz}}} \sup_{u \in K} \|u - D_n \circ E_n (u)\|_U.
  \end{equation*}
  In order to have a well-defined Lipschitz constant, we also optimize over the space $Y = \real^n$ and in particular its norm $\|\cdot\|_Y$.

  \item Finally, we remove the encoder in favor of an optimal encoding from a bounded set to arrive at the \emph{Lipschitz width} \cite{PetrovaWojtaszczyk2021}
  \begin{equation*}
    \lwidth_n(\kset) = \inf_{\substack{D_n, Y \\ \gamma-\text{Lipschitz}}} \sup_{u \in K} \inf_{\|y\| \le B} \|u - D_n(y)\|_U.
  \end{equation*}

\end{enumerate}
Instead of encoding $u \in K$ by finite dimensional vectors in $\real^n$, we can also encode it by bit sequences of length $n$. Each such sequence $y_i$ is decoded to a single function $u_i$ and hence a guaranteed approximation error $\min_i \|u - u_i\|_U \le \epsilon$ yields an $\epsilon$ cover of $K$. Hence, the best possible error is given by the \emph{metric entropy}
\[
  \entropy_n(K) = \inf \left\{ \epsilon \, \middle| \, K\text{ is covered by }2^n\text{ balls of radius }\epsilon \right\}.
\]
In contrast to manifold width, this benchmark does not provide any stability guarantees but is often easier to compute. See Section \ref{sec:benchmarks} for comparisons.

Although these width are well studied in approximation theory, a straightforward application to reduced order models along the lines of the Kolmogorov $n$-width is problematic.

\paragraph{Problem I: Internal Representation of Encoder and Decoders}

The Kolmogorov $n$-width measures how well we can approximate $u \in K$ from $n$ dimensional linear spaces $U_n$, but provides no information on how to encode or store the linear spaces themselves. For classical approximation methods, like splines, finite elements or wavelets, these spaces are hard coded and hence have negligible cost. For reduced order models, these spaces are defined by their bases $\psi_i$, which in turn are high fidelity finite element PDE solutions, or similar, and thus have non-trivial cost themselves. However, the storage cost for the basis $\psi_i$ can be safely ignored in view of an offline/online decomposition, which only needs derived quantities like stiffness matrices $a(\psi_i, \psi_j)$ or functionals $\ell_\mu(\psi_i)$ of size $\mathcal{O}(n^2)$ to approximate quantities of interest in the online phase.

This online/offline decomposition is applicable to all linear spaces $U_n$ in the definition of the Kolmogorov $n$-width, so that it provides an informative lower bound. On the other hand online/offline decompositions are not known for arbitrary nonlinear encoder/decoder pairs in the definition of the non-linear width or entropy, and hence, it is not clear how much extra storage cost we need for $E_n$ and $D_n$ adapted to the solution manifold $K = \{u_\mu | \mu \in \param\}$ of a parametric problem.

\paragraph{Problem II: Trivial decoders}

For application to the solution manifold, the nonlinear widths tend to be too general. Indeed, for parameter space $\param \subset \real^D$ and $n \ge D$, a trivial decoder is the solution operator itself 
\begin{align*}
  D_n: \mu & \to u_\mu, &
  & \text{s.t.} &
  F_\mu(u_\mu) & = 0.
\end{align*}
A corresponding encoder is then a parameter estimation problem $E_n: u_\mu \to \mu$. For elliptic problems, the decoder is smoothing, even analytic 
\cite{
  CohenDeVoreSchwab2010,
  CohenDeVoreSchwab2011,
  ChkifaCohenDeVoreSchwab2013,
  CohenDeVore2015%
}
, giving rise to very favorable performance of reduced basis and related methods. On the flip side, the encoder or parameter estimation problem tends to be ill-posed and therefore cannot directly serve in the definition of the manifold width (depending on the choice of norms).

For transport dominated problems, though, the situation is different. The forward map is much less benign, resulting in many difficulties and dedicated methods in the literature. On the flip side, the encoder is often continuous, as can be seen in the following simple linear Riemann problem.

\begin{example}
  Consider the linear Riemann problem
  \begin{align*}
    F_\mu(u_\mu) = u_t + \mu u_x & = 0, & 0 < t & < T, \, x \in \real, \\
    u_\mu(0, x) & = g(x), & t & = 0, \, x \in \real,
  \end{align*}
  with parameter $\mu \in \param := [-1, 1]$ and Heaviside function $g(x) = 1$ for $x \ge 0$ and $g(x) = 0$ else. The parameter to solution map is $\mu \to u_\mu(t,x) = g(x - \mu t)$ and Lipschitz in $L_1$. We can recover the parameter by integration
  \[
    \mu = 2 - 2 \int_0^1 \int_{-1}^1 g(x - \mu t) \, dx \, dt
  \]
  yielding a Lipschitz solution to parameter encoder.
  
\end{example}
Hence, for many problems, we can choose the solution map as encoder and the parameter estimation as decoder resulting in zero manifold width
\begin{align*}
  \mwidth_n(K) 
  = \mwidth_n^\gamma(K)
  = \mwidth_n^{\gamma,B}(K) & = 0, &
  & \text{for all} & 
  n & \ge D, &
  \param \subset \real^D.
\end{align*}
Of course this encoder/decoder pair is uninteresting for reduced order modeling because it involves full system solves. Indeed, the nonlinear widths ignore any computational time and space constraints and are therefore only partially informative.

\paragraph{Other Benchmarks}

Besides the manifold, stable and Lipschitz width, there are other benchmarks for nonlinear approximation, often tailored to specific classes of approximation methods. The nonlinear Kolmogorov width \cite{Temlyakov1998} provides optimal best $m$-term approximation, \cite{Donoho2001} considers optimal dictionary approximation and the $(M,N)$-width \cite{RimPeherstorferMandli2023} provides optimal bounds for approximation from transported subspaces.

\subsection{New Contributions}

As we have seen, a simple replacement of the Kolmogorov $n$-width with established nonlinear width entails several difficulties. However, they are avoided by a shift in perspective from the approximation of the solution manifold, or solution to parameter map, to the full reduced order model pipeline:
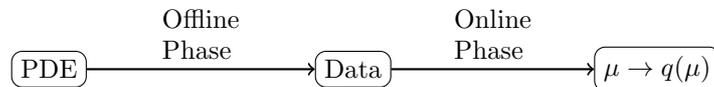
\begin{center}
  \begin{tikzpicture}

    \node[draw,rounded corners] (pde) at (0,0) {PDE};
    \node[draw,rounded corners] (data) at (4,0) {Data};
    \node[draw,rounded corners] (qi) at (8,0) {$\mu \to \qi(\mu)$};

    \draw[thick,->] (pde) -- (data) node[midway,above,align=left] {Offline\\Phase};
    \draw[thick,->] (data) -- (qi) node[midway,above,align=left] {Online\\Phase};

  \end{tikzpicture}
\end{center}
Given a PDE, the offline phase computes some data, usually reduced stiffness matrices and load vectors. Given this data, the online phase provides an approximation of the quantity of interest $\qi(\mu)$ for any parameter $\mu$. In this view, the offline and online phase naturally serve as an encoder/decoder pair from a class of parametric PDEs to the parameter to quantity of interest map. Note that the input to the decoder is a class of PDEs, e.g. all linear transport equations with controlled smoothness of inflow boundary condition, PDE coefficients and quantity of interest, not one single PDE as for the definition of the solution manifold.

In Section \ref{sec:benchmarks}, we use this new perspective to define manifold, stable and Lipschitz width. In Section \ref{sec:transport}, we apply them to the parametric transport PDE
\begin{align*}
  \qi(\mu) & = \int_{\Gamma_+} g_+ u_\mu, &
  b_\mu \nabla u_\mu & = 0, &
  u_\mu|_{\Gamma_{-}} & = g_-
\end{align*}
with $\sm$ smooth inflow boundary condition $g_-$, $\sp$ smooth quantity of interest $g_+$ and $\sb$ smooth flow field, depending on $D+d-2$ dimensional parameters. We show that the Lipschitz width satisfies the lower bound
\begin{align*}
  \mwidth_n^{\gamma,1}
  & = \Omega(n^{-\alpha}), &
  \alpha & < n^{-\frac{\sm+\sp+ (d-1)}{\max\{D/\sb, d-2\}}}
\end{align*}
for any $\alpha$ that satisfies the given constraint. For motivation and context, Sections \ref{sec:elliptic-rb} and \ref{sec:elliptic} contain some comparison to elliptic problems.

\subsection{Notations}

Throughout the paper, we use $\lesssim$, $\gtrsim$ and $\sim$ for less, bigger and equivalent up to constants that can change in every occurrence and are independent of dimensions and smoothness.

\section{Offline/Online Decomposition}
\label{sec:elliptic-rb}

Before we consider nonlinear width for the full reduced order model pipeline, we first briefly review the offline/online decomposition of a typical reduced basis method. We are interested in many parameter queries of a quantity of interest 
\begin{equation*} % \label{eq:rb-elliptic-1}
  \qi(\mu) = \ell(u_\mu)
\end{equation*}
for some functional $\ell \in U' := H^{-1}(\Omega)$ with solution $u_\mu \in U := H^1_0(\Omega)$ given by an elliptic variational problem
\begin{equation*} \label{eq:elliptic}
  \begin{aligned}
    a_\mu(u_\mu, v) & 
    = \sum_{q=1}^Q \theta_q(\mu) a_q(u_\mu, v)
    = f(v) & \text{for all } v & \in H^1_0(\Omega),
  \end{aligned}
\end{equation*}
with $f \in U'$ and an affine decomposition of the parameter dependence. Given a reduced basis
\[
  U_n = \operatorname{span} \{ \psi_1, \dots, \psi_n \} \subset U,
\]
we first approximate $u_\mu$ in $U_n$ with a Galerkin method and then evaluate the functional $\ell$ of the approximation. Using the linearity and affine decomposition, it is easy to see that this yields
\begin{align*}
  \ell(u_\mu) & \approx \boldsymbol{\ell}^T \boldsymbol{u}_\mu, &
  \left(\sum_{q=1}^Q \theta_q(\mu) A_q\right) \boldsymbol{u}_\mu & = \boldsymbol{f},
\end{align*}
or equivalently
\begin{equation} \label{eq:elliptic-online}
  \ell(u_\mu) \approx \boldsymbol{\ell}^T \left( \sum_{q=1}^Q \theta_q(\mu) A_q \right)^{-1} \boldsymbol{f}
\end{equation}
with stiffness matrix and vectors
\begin{align} \label{eq:online-store}
  (A_q)_{ij} & = a_q(\psi_i, \psi_j), &
  \boldsymbol{f}_i & = f(\psi_i), &
  \boldsymbol{\ell}_i & = \ell(\psi_i), &
  i,j & = 1, \dots, n.
\end{align}
In particular, the basis functions $\psi_i$ themselves are not necessary to approximate the output $\mu \to \ell(u_\mu)$, only the $2n + Qn^2$ numbers in \eqref{eq:online-store}. Unravelling the inverse matrix with Cramer's rule, it is easy to see that this is a rational approximation of $\ell(u_\mu)$ in terms of $\theta_q(\mu)$ of degree $(n-1,n)$ in numerator and denominator. Thus, instead of asking how well we can approximate $u_\mu$ on the solution manifold by the Kolmogorov $n$-width, we can ask how good is the rational approximation of $\mu \to \ell(u_\mu)$?

\section{Benchmarks for Reduced Order Models}
\label{sec:benchmarks}

In the last section, we have seen that the online/offline decomposition of the reduced basis method provides a (rational) approximation of the quantity of interest $\qi(\mu) = \ell(u_\mu)$ from some stored matrices and vectors. How good is this approximation, or generally $\qi_{\text{reduced basis}}(\mu)$, $\qi_{\text{POD}}(\mu)$, $\qi_{\text{neural network}}(\mu)$, etc. versus some optimal benchmark? To obtain a meaningful question, we need to specify a norm in which we measure the error, a class of parametric PDEs for which we compute reduced order models and a precise definition what we consider as reduced order model to enter the competition.

\paragraph{Error Space}

We assume that the quantity of interest $\mu \to \qi(\mu)$ is contained in some Banach space $Z$. The space $Z = L_2(\param)$ is a natural choice for POD and $Z = L_\infty(\param)$ for reduced basis methods, which provide point-wise error bounds $|\qi(\mu) - \qi_{\text{reduced basis}}(\mu)|$. Note that $\mu$ is an independent variable and we ask how well we can approximate the function $\mu \to \qi(\mu)$, whereas the Kolmogorov $n$-width considers the difference $\inf_{u_n \in U_n} \|u_\mu - u_n\|_U$ for any fixed $\mu$. 

\paragraph{Compact Model Class $K$}

As for any other approximation method, to obtain meaningful error bounds, we need the quantities of interest $\qi \in Z$ to be confined to a compact set $K \subset Z$. Since we are interested in PDE solutions
\begin{align} \label{eq:ppde-class}
  \qi \colon \param & \to \real, &
  \qi(\mu) & = \ell_\mu(u_\mu), &
  & \text{s.t.} & 
  F_\mu(u_\mu) & = 0, 
\end{align}
for functional $\ell_\cdot\colon \param \times U \to \real$, given by $(\mu, u) \to \ell_\mu(u)$, and differential operator $F_\cdot\colon \param \times U \to V$, given by $(\mu, u) \to F_\mu(u)$, it is natural to define
\[
K := \{\mu \to \qi(\mu) = \ell_\mu(u_\mu) \,|\, F_\mu(u_\mu) = 0, \, (\ell, F) \in \problems \} \subset Z
\]
for a class  $\problems$ of parametric PDEs. Since the PDEs are typically given by initial or boundary conditions, PDE coefficients, etc., it is convenient to combine these into a secondary parameter $\instance \in \instances$ that defines each problem in $\problems$, i.e.
\[
\problems
= \{ (\ell_{\cdot, \instance}, F_{\cdot, P}) | P \in \instances \},
\]
for some $\ell_{\mu, \instance}$ and $F_{\mu, \instance}$ that depend on the parameter $\mu$ and a secondary parameter $\instance \in \instances$ to select a quantity of interest and PDE from $\problems$.

\begin{example} \label{ex:elliptic}
  Consider a class of elliptic PDEs, given in weak form by
  \begin{equation} \label{eq:1:ex:elliptic-2}
    \begin{aligned}
    \dualp{F_\mu(u), v} := (a_\mu \nabla u, \nabla v) - \dualp{f,v} & = 0 & & \text{for all } v \in H^1_0(\Omega) \\
      \qi(\mu) & = \ell(u_\mu)
    \end{aligned}
  \end{equation}
  with domain $\Omega \subset \real^d$, right hand side and quantity of interest $f, \ell \in H_0^{-1}(\Omega)$ and affine diffusion
  \begin{equation} \label{eq:2:ex:elliptic-2}
    a_\mu(x) = \bar{a}(x) + \sum_{q=1}^Q \mu_q \varphi_q(x),
  \end{equation}
  In order to ensure ellipticity, we assume that
  \begin{align*}
    2 & \le \bar{a} \le 3, &
    \sum_{1=q}^Q |\varphi_q(x)| & \le 1, &
    \mu & \in \param := [-1, 1]^Q.
  \end{align*}
  Then, we define the problem class
  \[
    \problems := \left\{ (\ell, F) \middle| \, F\text{ satisfies \eqref{eq:1:ex:elliptic-2}, \eqref{eq:2:ex:elliptic-2} with }\|f\|_{H^{-1}(\Omega)} \le 1, \, \|\ell\|_{H^{-1}(\Omega)} \le 1 \right\}.
  \]
  For fixed $\bar{a}$ this class has the natural parametrization $\instances$ 
  \begin{equation} \label{eq:elliptic-problem-class}
    \instances := \{\varphi \in B_{L_\infty}^Q\} \times \{f \in B_{H^{-1}}\} \times \{\ell \in B_{H^{-1}}\},
  \end{equation}
  where $B_\Box$ denotes a unit ball in space $\Box$. The first component contains the $Q$ functions $\varphi_q$, the second the right hand side $f$ and the third the quantity of interest $\ell$. 

\end{example}

The following example is of primary interest and will be discussed more carefully in Section \ref{sec:transport}.

\begin{example} \label{ex:transport}
  
  Consider the linear transport problems
  \begin{align}
    \qi(\mu) & = \ell(u_\mu) = \int_{\Gamma_+} g_+ u_\mu, &
    F_\mu(u_\mu) & = \left[ b_\mu \nabla u_\mu - f, \,\,\,\,  u_\mu|_{\Gamma_-} - g_- \right]
    = [0, 0]
  \end{align}
  with a given flow field $b_\mu$, right hand side $f$, inflow and outflow boundaries $\Gamma_-$ and $\Gamma_+$, initial value $g_-$ and quantity of interest $g_+$ on the outflow boundary. If we choose a parametric flow field $b_\mu$ for which the PDE has unique solutions, we obtain the class
  \[
    \problems = \left\{ (\ell, F) \, \middle| \,  \|g_-\|_{W^{\sm,p}(\Gamma_-)} \le 1, \, \|g_+\|_{W^{\sp,\infty}(\Gamma_+)} \le 1 \right\}
  \]
  parametrized by
  \begin{equation}
      \instances(b) = \{g_- \in B_{W^{\sm,p}(\Gamma_-)}\} \times \{g_+ \in B_{W^{\sp,\infty}(\Gamma_+)}\} 
  \end{equation}
  where again $B_\Box$ denotes a unit ball in $\Box$. Classes with variable $b$ are considered in Section \ref{sec:transport}.
\end{example}

In the reminder of the paper, we will only work with the parametrization $\instances$ of the PDE and the quantity of interest $\qi \in K \subset Z$. The particular structure \eqref{eq:ppde-class} how $\instance \in \instances$ gives rise to the quantity of interest $\qi$ is not important and therefore we generalize it to the map
\begin{align*}
  \Qi: \instances & \to Z, &
  \Qi(\instance) & = \mu \to \qi(\mu) = q(\cdot)
\end{align*}
that maps a problem instance, to the quantity of interest. Since we only work with the parametrization $\instances$, we do not distinguish it from $\problems$ and call both a \emph{problem class}. We also  refer to their respective elements $\instance \in \instances$ and $(\ell, F) \in \problems$ as a \emph{problems instance}. In conclusion, our compact set $K$ of quantities of interest becomes
\[
  K = \Qi(\instances) \subset Z.
\]

Roughly speaking, one may think of a problem class as all problems that can be addressed by one software implementation. There are many software packages for Example \ref{ex:elliptic} and others for Example \ref{ex:transport}, but usually not for both simultaneously as they have different parametrizations and stability conditions.

\paragraph{Reduced Order Models}

For the purpose of our benchmark, a \emph{reduced order model} consists of two functions 
\begin{equation*}
  \begin{aligned}
    \offfunc_n: & & \instances & \to \real^n \\
    \onfunc_n: & & \real^n & \to Z.
  \end{aligned}
\end{equation*}
The first function represents the offline phase and computes all data for later reconstruction combined into one single vector in $\real^n$, e.g. the matrices and vectors in \eqref{eq:online-store}. The second function represents the online phase and uses the stored data to compute an approximation of the parameter to quantity of interest map $\mu \to \qi(\mu) \in Z$. 

In the following, we consider the worst case performance of the reduced order model, i.e. the smallest $\epsilon > 0$ so that 
\begin{align*} % \label{eq:rom-error}
  \|\Qi(P) - \onfunc_n \circ \offfunc_n (P)\|_Z & \le \epsilon, & 
  \forall P \in \instances. 
\end{align*}

\begin{example}

  For the reduced basis method in Section \ref{sec:elliptic-rb} and Example \ref{ex:elliptic}, we have:
  \begin{itemize}
    \item $\offfunc_n$: Given a problem instance $\instance = (\varphi, \ell, f)$ and a dimension $n$ choose the smallest dimension $m$ of the reduced basis so that $2m + m^2 \le n$. Compute the reduced basis $\psi_1, \dots, \psi_m \subset U$ with the greedy method and output the corresponding stiffness matrix and vectors in \eqref{eq:online-store} consisting of less than $n$ numbers by the choice of $m$.
    \item $\onfunc_n$: Given the output of the offline phase, approximate $\qi(\mu)$ by \eqref{eq:elliptic-online}.
  \end{itemize}

\end{example}

\paragraph{Nonlinear Width}

For any reduced order method as defined above, the offline and online phase naturally have the structure of an encoder and decoder. Hence, we can use them to define nonlinear width to benchmark reduced order methods. 

\begin{enumerate}
  \item Optimizing over all continuous $\offfunc_n$/$\onfunc_n$, we obtain the manifold width
  \[
    \mwidth_n(\instances) = \inf_{\substack{\offfunc_n \\ \onfunc_n \\ \text{continuous}}} \sup_{\instance \in \instances} \|\Qi(\instance) - \onfunc_n \circ \offfunc_n(\instance) \|_Z.
  \]
  \item Optimizing over all $\gamma$-Lipschitz encoder/decoders and norms $\|\cdot\|_Y$ for $\real^n$, we obtain the stable manifold width
  \[
    \mwidth_n^\gamma(\instances) = \inf_{\substack{\offfunc_n \\ \onfunc_n, Y \\ \gamma-\text{Lipschitz}}} \sup_{\instance \in \instances} \|\Qi(\instance) - \onfunc_n \circ \offfunc_n(\instance) \|_Z.
  \]
  \item With bounded output $\ondata = \offfunc_n(P)$ of the offline phase, we obtain the Lipschitz width
  \[
    \mwidth_n^{\gamma,B}(\instances) = \inf_{\substack{\onfunc_n, Y \\ \gamma-\text{Lipschitz}}} \sup_{\instance \in \instances} \inf_{\|\ondata\| \le B} \|\Qi(\instance) - \onfunc_n(\ondata) \|_Z.
  \]
\end{enumerate}

\begin{remark}
  All three width require a norm or metric on $\instances$ to properly define (Lipschitz) continuity. This provides another justification to replace the problem class $\problems$ with its parametrization $\instances$ because the latter typically has a natural metric as in Examples \ref{ex:elliptic} and \ref{ex:transport}.
\end{remark}

\begin{remark}
  The manifold and stable manifold widths are identical to the definitions in Section \ref{sec:nonlinear-width}, with one exception: The domain of the encoder is a parametrization $\instances$ of the class $\problems$ of parametric PDEs instead of functions in the class $K = \Qi(\instances)$. On the other hand, the definition of the Lipschitz width does not use the encoder and therefore carries over verbatim.
\end{remark}

An alternative benchmark, is the metric entropy of the class $K = \Qi(\instances)$ given by
\[
  \entropy_n(\Qi(\instances)) = \entropy_n(K) = \inf \left\{ \epsilon \middle| K\text{ is covered by }2^n\text{ balls of radius }\epsilon \right\}.
\]
While the widths measure how many real numbers we need to store in order to approximate $\qi \in \Qi(\instances) \subset Z$, the entropy measures how many bits we need for approximation. Indeed, we can equivalently define it similar to the Lipschitz width as
\[
  \bentropy_n(\Qi(\instances))
  = \bentropy_n(K) 
  = \inf_{D} \sup_{q \in K} \inf_{b \in \{0,1\}^n} \|q - D(b)\|,
\]
where the infimum is over all decoders that map bit sequences $\ondata \in \{0,1\}^n$ of length $n$ to approximations, without stability guarantees.

\begin{lemma}
  For any compact set $K$, we have $\entropy_n(K) = \bentropy_n(K)$.
\end{lemma}

\begin{proof}

Let $q_i$, $i=1, \dots, 2^n$ be an optimal $\epsilon$ cover in the definition of the entropy $\entropy_n(K)$. Let $b \in \{0,1\}^n$ be the binary representation of $i$ and define $D(b) := q_i$. Then, by construction $\inf_{b \in \{0,1\}^n} \|q - D(b)\| \le \epsilon$ and thus $\bentropy_n(K) \le \entropy_n(K)$.

In the other direction, let $D: \{0,1\}^n$ be an optimal decoder in the definition of $\bentropy_n(K)$. For every bit sequence $b \in \{0,1\}^n$, let $i$ be the corresponding integer number and $q_i := D(b)$. Then $q_i$ builds an $\bentropy_n(K)$-cover of $K$ and thus $\entropy_n(K) \le \bentropy_n(K)$.
  
\end{proof}

\paragraph{Comparison of Benchmarks}

The recent papers \cite{CohenDeVorePetrovaWojtaszczyk2022,PetrovaWojtaszczyk2021,PetrovaWojtaszczyk2022} establish connections between the different widths and entropy. Several of the results directly carry over to our slight modifications for reduced order modeling. Since the definitions of entropy and Lipschitz width are unchanged from Section \ref{sec:nonlinear-width}, the paper \cite{PetrovaWojtaszczyk2021} shows that they have comparable asymptotics, if polynomial or exponential.

\begin{theorem}[{\cite[Corollary 4.8]{PetrovaWojtaszczyk2021}}] \label{th:entropy-lipwidth}
  Let $Z$ be a Banach space, let $K = \Qi(\instances) \subset Z$ be compact and assume that $\gamma$ is bigger than the diameter of $K$. Then
  \begin{enumerate}
    \item \emph{Polynomial rates:} For $\alpha > 0$ and $\beta \in \real$
    \begin{align*}
      \entropy_n(\Qi(\instances)) & = \mathcal{O} \left( \frac{[\log_2 n]^\beta}{n^\alpha} \right) &
      & \Rightarrow &
      \mwidth_n^{\gamma,1}(\instances) & = \mathcal{O} \left( \frac{[\log_2 n]^\beta}{n^\alpha} \right) &
      \\
      \entropy_n(\Qi(\instances)) & = \Omega \left( \frac{[\log_2 n]^\beta}{n^\alpha} \right) &
      & \Rightarrow &
      \mwidth_n^{\gamma,1}(\instances) & = \Omega \left( \frac{[\log_2 n]^\beta}{n^\alpha [\log_2 n]^\alpha} \right)
    \end{align*}

    \item \emph{Exponential rates:} For $0 < \alpha < 1$ and constants $c, c' > 0$
    \begin{align*}
      \entropy_n(\Qi(\instances)) & = \mathcal{O} \left( 2^{-c n^\alpha} \right) &
      & \Rightarrow &
      \mwidth_n^{\gamma,1}(\instances) & = \mathcal{O} \left( 2^{-c n^\alpha} \right)
      \\
      \entropy_n(\Qi(\instances)) & = \Omega \left( 2^{-c n^\alpha} \right) &
      & \Rightarrow &
      \mwidth_n^{\gamma,1}(\instances) & = \Omega \left( 2^{-c' n^{\alpha/(1-\alpha)}} \right)
    \end{align*}

  \end{enumerate}
  
\end{theorem}

Since the Lipschitz width does not pose any restrictions on the encoder, it is simple to compare Lipschitz and stable width.

\begin{theorem}[{\cite[Theorem 6.1]{PetrovaWojtaszczyk2021}}]

  Let $\Qi(\instances)$ be compact, $n \ge 1$ and $\gamma \ge 0$. Then
  \[
    \mwidth_n^{\gamma^2 \diam(\instances), 1}(\instances)
    \le \mwidth_n^\gamma(\instances).
  \]
  
\end{theorem}

\begin{proof}

The only difference to the standard width is that the domain of the encoder is $\instances$ instead of $K = \Qi(\instances)$. This does not change the proof in \cite[Theorem 6.1]{PetrovaWojtaszczyk2021}, which carries over verbatim.
  
\end{proof}

\paragraph{Applications}

In this paper, we are mainly interested in lower performance bounds for transport dominated parametric PDEs. In Section \ref{sec:transport}, we show results for the entropy. By the last two theorems, this directly implies lower bounds for the Lipschitz and stable width. Indeed, if
\[
  \entropy_n(\Qi(\instances)) \gtrsim n^{-\alpha}
\]
then for $\gamma \ge \diam(K)$ we have
\begin{align*}
  \mwidth_n^{\gamma, 1} & \gtrsim n^{-\bar{\alpha}}, &
  \mwidth_n^{(\gamma / \diam(\instances))^{1/2}} & \gtrsim n^{-\bar{\alpha}}.
\end{align*}
for all $\bar{\alpha} < \alpha$.

\paragraph{Comparison with Kolmogorov \texorpdfstring{$n$}{n}-width}

In the discussion above, we impose no restrictions on $\onfunc_n$ and $\offfunc_n$ other than (Lipschitz) continuity. This allows lower bounds for a large variety of methods, including classical reduced basis methods, but also alternative techniques such as neural networks. If, on the other hand, we do have the additional structure of a reduced basis method, the Lipschitz width provides a lower bound for the Kolmogorov $n$-width of the solution manifold. To this end, we first consider a relatively broad definition of a reduced basis method.
\begin{enumerate}[label=(RB\arabic*)]

  \item \label{item:rb-problems} \emph{Problem class:} Assume every problem $(\ell, F_\mu) \in \problems$ consists of a parametric operator $F_\mu: U \to V$ for $\mu \in \param$ and two Banach spaces $U$ and $V$ and a linear functional $\ell$ with $\|\ell\|_{U'} \le 1$. They define the quantity of interest $\qi_\mu = \ell(u_\mu)$, where $u_\mu$ solves $F_\mu(u_\mu) = 0$.

  \item \label{item:rb-online} \emph{Reduced basis method:} Assume we have a reduced solution operator $G_\mu: \ondata \to \real^m$ from the reduced data to reduced basis coefficients so that for every reduced basis $U^m = \linspan\{\psi_1, \dots, \psi_m\}$ the reconstruction 
  \[
    u_\mu^m := \sum_{i=1}^m G_\mu(\ondata)_i \psi_i
  \]
  is near optimal
  \begin{equation*}
    \|u_\mu^m - u_\mu\|_U \le c \inf_{u^m \in U^m} \|u^m - u_\mu\|_U,
  \end{equation*}
  $G_\mu$ is $\gamma$-Lipschitz in $\mu$ and the $\ondata$ is bounded $\|\ondata\| \le B \subset \real^n$ with $n \le \bar{c} m^\beta$ for constants $c, \bar{c}, \beta \ge 0$. Furthermore, the $\ondata$ contains the vector $\boldsymbol{\ell} := [\ell(\psi_i)]_{i=1}^m$.

\end{enumerate}
This definition is an abstraction of the standard reduced basis method in Section \ref{sec:elliptic-rb}. Indeed, for the latter we have
\begin{itemize}
  \item Arbitrary reduced basis $\psi_1,\dots, \psi_m$, 
  \item \ondata: Stiffness matrices and load vectors in \eqref{eq:online-store} of size $Qm^2+2m \sim cn^2$.
  \item Solution operator: Galerkin solution $G_\mu(\ondata) = \left(\sum_{q=1}^Q \theta_q(\mu) \boldsymbol{A}_q\right)^{-1} \boldsymbol{f}$.
  \item  Note: $G_\mu$ is independent of $\psi_i$, given \ondata.
  \item Note: $u_\mu^m$ is optimal (up to constants) by C\'{e}a's lemma.
\end{itemize}
Given the full reduced order model, we obtain the following comparison between Lipschitz width and Kolmogorov $n$-width of the solution manifold.

\begin{proposition}

  Assume that \ref{item:rb-problems} and \ref{item:rb-online} hold. Let $\instances$ be a parametrization of $\problems$ and $\mathcal{M}_\instance = \{u_\mu | \, \mu \in \param, \, F_\mu(u_\mu) = 0 \}$ be the solution manifold for problem instance $(\ell, F) \in \problems$, parametrized by $\instance \in \instances$. Let $Z = L_\infty(\param)$. Then
  \begin{equation*}
    \mwidth_{\bar{c} m^\beta} ^{\gamma, B}(\instances) \le c \sup_{\instance \in \instances} \nwidth_m(\mathcal{M}_\instance).
  \end{equation*}

\end{proposition}

Since the dimension $n$ of $\ondata$ and the dimension $m$ of the reduced basis are of different order, typically $n \sim m^2 = m^\beta$, the convergence rates of Lipschitz width and Kolmogorov $n$-width in the last result differ. Hence it is not clear if the comparison is sharp.

\begin{proof}

Let $\boldsymbol{\ell} := [\ell(\psi_i)]_{i=1}^m$. We define a decoder by $D_\mu = \boldsymbol{\ell}^T G_\mu(\ondata)$. By assumption on $\ell$ and $G_\mu$, the decoder is $\gamma$-Lipschitz uniformly in $\mu$ and therefore also $\gamma$-Lipschitz in the supremum norm $Z=L_\infty(\param)$. By the optimality assumption in \ref{item:rb-online}, we have
\begin{multline*}
  |\ell(u_\mu) - D_\mu(\ondata)|
  = \left| \ell \left( u_\mu - \sum_{i=1}^m G_\mu(\ondata)_i \psi_i \right) \right|
  \\
  \le \left\| u_\mu - \sum_{i=1}^m G_\mu(\ondata)_i \psi_i \right\|
  = \left\| u_\mu - u_\mu^m \right\|
  \le c \inf_{u^m \in U^m} \left\| u_\mu - u^m \right\|.
\end{multline*}
Since the bound is uniform in $\mu$, this implies
\begin{equation*}
  \inf_{\substack{y \in \real^m \\ \|y\| \le B}} \|\ell(u_\cdot) - D_\cdot(y)\|_Z
  \le \sup_{\mu \in \param} |\ell(u_\mu) - D_\mu(\ondata)| 
  \le c \sup_{\mu \in \param} \inf_{u^m \in U^m} \left\| u_\mu - u^m \right\|.
\end{equation*}
Note that given $\ondata$, including $\boldsymbol{\ell}$, the decoder is independent of the reduced basis $U^m$. Thus, taking the infimum over the linear space $U^m \subset U$ of dimension $m$ and then the supremum over arbitrary problems $\instance \in \instances$, implies
\begin{equation*}
  \sup_{\instance \in \instances} \inf_{\substack{y \in \real^n \\ \|y\| \le B}} \|\ell(u_\cdot) - D_\cdot(y)\|_Z
  \le c \sup_{\instance \in \instances} \nwidth_m(\mathcal{M}_\instance)
\end{equation*}
Taking the infimum over all decoders and norms $\|\cdot\|_Y$ on $\real^n$ provides
\begin{equation*}
  \mwidth_n^{\gamma, B}(\instances) \le c \sup_{\instance \in \instances} \nwidth_m(\mathcal{M}(\instance))
\end{equation*}
so that $n \le \bar{c} m^\beta$ and the monotonicity of $\mwidth_m^{\gamma, B}$ shows the result.

\end{proof}

The assumptions \ref{item:rb-online} are well-known for the reduced basis method applied to elliptic parametric problems. For transport dominated problems they are less clear because of stability issues for Galerkin methods 
\cite{
  RozzaVeroy2007,
  GernerVeroy2012,
  DahmenPleskenWelper2014,
  Dahmen2015%
}
. Nonetheless, a lower bound for the Lipschitz width entails either a lower bound for the Kolmogorov $n$-width, or if the latter is much better, that we cannot build a reduced basis method with the properties that we are used to from the elliptic case.

\section{Application to Elliptic Problems}
\label{sec:elliptic}

Although we are mainly interested in transport problems, for the sake of completeness, we mention some well-known upper bounds for elliptic problems. In this case, the map $\Qi: \instances \to K$ of problem class parametrizations to quantity of interest is smoothing. This allows compact $K$ with excellent performance of reduced order models and corresponding small entropies, even if the problem classes' parametrizations $\instances$ are non-compact. For Example \ref{ex:elliptic}, we obtain entropy bounds based on the following result:

\begin{theorem} [{\cite[Section 2.5]{CohenDeVore2015}}]
  For the problem class \eqref{eq:elliptic-problem-class}, all functions in $K = \Qi(\instances)$ are analytic and uniformly bounded in an extended polydisc $L_\infty([-r,r]^Q)$ for some $r>1$.
\end{theorem}

Without loss of generality, we may assume that the $L_\infty$ bound of $K$ is one. Then, \cite[p.505, Lemma 5.5]{LorentzGolitschekMakovoz1996} yields the following entropy bound, where
\[
  H_\epsilon(K) = \inf \left\{ n \, \middle| \, K\text{ is covered by }2^n\text{ balls of radius }\epsilon \right\}.
\]
is defined as the entropy, but minimized over the size $n$ of the cover instead of the radius $\epsilon$.

\begin{corollary}
  For the problem class \eqref{eq:elliptic-problem-class}
  \begin{equation*}
    H_\epsilon(\Qi(\instances)) \lesssim \log^{Q+1} \left(\frac{1}{\epsilon}\right).
  \end{equation*}
\end{corollary}

In fact, the papers
\cite{
  CohenDeVoreSchwab2010,
  CohenDeVoreSchwab2011,
  ChkifaCohenDeVoreSchwab2013,
  CohenDeVore2015%
} 
contain a much more careful analysis that allows infinite dimensional parameters $Q = \infty$.

\section{Application to Transport Problems}
\label{sec:transport}

In this section, we establish upper and lower bounds for the entropy of parametric linear transport PDEs.

\subsection{Main Result}

We consider entropy bounds for linear stationary transport problems
\begin{equation} \label{eq:transport-class}
  \begin{aligned}
    b_\mu \cdot \nabla u_\mu & = 0, &  & \text{in }\Omega \subset \real^d \\
    u_\mu|_{\Gamma_-} & = g_-, & & \text{on }\Gamma_- \subset \partial \Omega,
  \end{aligned}
\end{equation}
in $d \ge 2$ dimensions with quantity of interest
\begin{equation*}
  \qi(\mu) := \ell(u_\mu) := \int_{\Gamma_+} g_+ u_\mu.
\end{equation*}
The PDE's right hand side is set to zero, because it leaves the entropy invariant as argued in Section \ref{sec:rhs-entropy}. As usual, the boundary $\partial \Omega = \Gamma_- \cup \Gamma_0 \cup \Gamma_+$ is split into inflow, characteristic and outflow boundaries
\begin{align*}
  \Gamma_- := \Gamma_-(b_\mu) := \{x \in \partial \Omega | \, b_\mu(x) \cdot n(x) < 0\}, \\
  \Gamma_0 := \Gamma_0(b_\mu) := \{x \in \partial \Omega | \, b_\mu(x) \cdot n(x) = 0\}, \\
  \Gamma_+ := \Gamma_+(b_\mu) := \{x \in \partial \Omega | \, b_\mu(x) \cdot n(x) > 0\}.
\end{align*}
Since these are parameter dependent, we define $g_-$ and $g_+$ on the entire boundary $\partial \Omega$ so that all inflows and quantities of interest are well-defined.

Existence and uniqueness for the transport problem \eqref{eq:transport-class} can be delicate, for example if the characteristic equations have periodic solutions that never reach the boundary \cite{Evans2010}. Therefore, we make the following the minimal assumption on the flow fields to ensure that all quantities of interest are well defined.
\begin{assumption} \label{assumption:pde-solution}
  Let $\flows(g_+)$ be the set of all parametric flow fields $b_\mu$ for which the transport equation \eqref{eq:transport-class} has a unique solution on a subdomain $\omega \subset \Omega$ that includes the quantity of interest $\supp(g_+) \subset \omega$, for all parameters $\mu \in \param$.
\end{assumption}
Intuitively, this means that all characteristics passing through the support of $g_+$ also intersect the inflow boundary. All remaining characteristics do not influence the quantity of interest and are therefore of no interest to us.

For $0 < \sm, \sp$, $1\le \sb$ and $1 \le p < \infty$ and some sufficiently large constant $C_b$ (dependent on the smoothness of $\Xi$ in Assumption \ref{assumption:rdom:reference-domain} below), we consider two different problem classes parametrized by:
\begin{equation} \label{eq:transport-class-parametrization-inflow}
  \instances(b, g_+) := \left\{ g_- : \, g_- \in B_{W^{\sm,p}(\partial \Omega)} \right\}
\end{equation}
and
\begin{equation} \label{eq:transport-class-parametrization}
  \begin{multlined}
    \instances := 
      \left\{ (g_-, g_+, b) : \, g_- \in B_{W^{\sm,p}(\partial \Omega)}, \, g_+ \in B_{W^{\sp,\infty}(\partial \Omega)}, \, \right. \\
      \left. b \in B_{C^{\sb}(\Omega \times \param)}(C_b) \cap \flows(g_+) \right\},
  \end{multlined}
\end{equation}
where $B_\Box = B_\Box(1)$ and $B_\Box(r)$ denotes the ball of radius $r$ in the function space $\Box$. The first one contains only one fixed transport field $b$ and has multiple smooth inflow boundary conditions and outflow quantities of interest. The second one expands this to arbitrary smooth transport fields. For both classes, the error of the quantity of interest is measured in the $Z = L_\infty(\param)$-norm. 

In our setup, only the flow field $b_\mu$ depends on the parameter and therefore strongly influences the entropy numbers. Indeed, if it is parameter independent, the quantities of interest are so as well and the entropy is negligible. Therefore, in order to obtain meaningful bounds we need some control over the parameter dependence in $b$. This is achieved by a flow field defined on a reference domain, which is then mapped into the actual domain $\Omega$ by the following map.

\begin{assumption} \label{assumption:rdom:reference-domain}
  There is a map $\rdom$ with
  \begin{enumerate}
    \item $\rdom \colon [-1,1] \times [-1,1]^{d-1} \to \Omega$
    \item $\rdom$ is one-to-one.
    \item $\{\rdom(0,w): \, w \in [-1,1]^{d-1}\} =: \bi \subset \partial \Omega$
    \item $\{\rdom(1,w): \, w \in [-1,1]^{d-1}\} =: \bo \subset \partial \Omega$
    \item $\|\rdom\|_{C^{\sb}([-1,1] \times [-1,1]^{d-1})} \lesssim 1$, $\|\rdom^{-1}\|_{C^{\sb}(\rdom([-1,1] \times [-1,1]^{d-1}))} \lesssim 1$, 
  \end{enumerate}
\end{assumption}
Strongly parameter dependent flow fields are constructed from $\partial_1 \rdom$ with an extra parametric drift in in Section \ref{sec:transport:reference-domain}. The sets $\bi$ and $\bo$ will be subsets of the inflow and outflow boundaries $\Gamma_-$ and $\Gamma_+$.  An illustration is in Figure \ref{fig:reference-domain}. We obtain the following entropy bounds, first for fixed flow field.

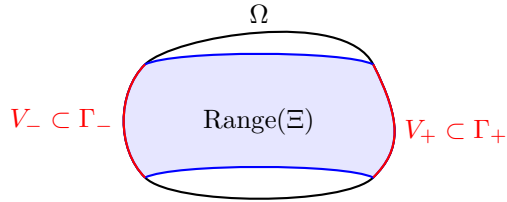
\begin{figure}[htb]

  \begin{center}
    \begin{tikzpicture}[scale=0.75]

      \draw [thick]
            (0,0) 
                  .. controls (0.5,-0.5) and (3.5,-0.5) .. 
            (4,0)
                  .. controls (4.5, 0.5) and (4.5, 1) .. 
            (4,2) 
                  .. controls (3.5, 3) and (0.5, 2.5) .. 
            (0,2) node[midway, above] {$\Omega$}
                  .. controls (-0.5, 1.5) and (-0.5, 0.5) ..
            (0,0);

      \draw [thick,blue,fill=blue!10]
            (0,0) 
                  .. controls (0.5,0.25) and (3.5,0.25) .. 
            (4,0)
                  .. controls (4.5, 0.5) and (4.5, 1) .. 
            (4,2) 
                  .. controls (3.5, 2.25) and (0.5, 2.25) .. 
            (0,2) 
                  .. controls (-0.5, 1.5) and (-0.5, 0.5) ..
            (0,0);

      \draw [thick,red]
            (0,2) 
                  .. controls (-0.5, 1.5) and (-0.5, 0.5) ..
            (0,0) node[midway,left] {$\bi \subset \Gamma_-$};

      \draw [thick,red]
            (4,0)
                  .. controls (4.5, 0.5) and (4.5, 1) .. 
            (4,2) node[midway,right] {$\bo \subset \Gamma_+$};

      \draw (2,1) node {$\operatorname{Range}(\Xi)$};

    \end{tikzpicture}
  \end{center}

  \caption{Illustration of Assumption \ref{assumption:rdom:reference-domain}.}
  \label{fig:reference-domain}
  
\end{figure}

\begin{theorem} \label{th:transport-entropy}

  Let the problem class with parametrization $\instances(b, g_+)$ be defined by \eqref{eq:transport-class-parametrization-inflow},  let Assumption \ref{assumption:rdom:reference-domain} be satisfied and $\param$ be an open parameter set of dimension larger than $d-1$. Then there is a flow field $b \in C^{\sb}(\Omega \times \param) \cap \flows(g_+)$ and outflow quantity of interest $g_+ \in B_{W^{\sp,\infty}(\partial \Omega})$ with
  \begin{equation*}
    \entropy_n(\Qi(\instances(b, g_+))) \gtrsim n^{-\frac{\sm+\sp+ (d-1)}{d-1}}.
  \end{equation*}
  
\end{theorem}

Entropy bounds for variable flow field are given in the next result.

\begin{theorem} \label{th:transport-entropy-var}

  Let the problem class with parametrization $\instances$ be defined by \eqref{eq:transport-class-parametrization},  let Assumption \ref{assumption:rdom:reference-domain} be satisfied and $\param$ be an open parameter set of dimension larger than $d-2 + D$, for some $D \ge 0$. Then
  \begin{equation*}
    \entropy_n(\Qi(\instances)) \gtrsim n^{-\frac{\sm+\sp+ (d-1)}{\max\{D/{\sb}, d-2\}}}.
  \end{equation*}
  
\end{theorem}

With Theorem \ref{th:entropy-lipwidth} or \cite[Corollary 4.8]{PetrovaWojtaszczyk2021}, the entropy lower bounds imply bounds for the Lipschitz width.

\begin{corollary} \label{cor:transport-entropy}
   
  Let all assumptions of Theorem \ref{th:transport-entropy} be satisfied. Then there is a flow field $b \in C^{\sb}(\Omega \times \param)$ and outflow quantity of interest $g_+ \in B_{W^{\sp,\infty}(\partial \Omega})$ so that for $\gamma > \diam(\Qi(\instances(b, g_+)))$ and any $\alpha$ that satisfies the condition below
  \begin{align*}
    \mwidth_n^{\gamma,1}(\instances(b, g_+))
    & = \Omega(n^{-\alpha}), &
    \alpha & <\frac{\sm+\sp+ (d-1)}{d-1}.
  \end{align*}
  
\end{corollary}

\begin{corollary} \label{cor:transport-entropy-var}
   
  Let all assumptions of Theorem \ref{th:transport-entropy-var} be satisfied and $\gamma > \diam(\Qi(\instances))$. Then for any $\alpha$ that satisfies the condition below
  \begin{align*}
    \mwidth_n^{\gamma,1}(\instances)
    & = \Omega(n^{-\alpha}), &
    \alpha & < n^{-\frac{\sm+\sp+ (d-1)}{\max\{D/\sb, d-2\}}}.
  \end{align*}
  
\end{corollary}

\begin{proof} [Proof of Corollaries \ref{cor:transport-entropy} and \ref{cor:transport-entropy-var}]

Both corollaries follow directly from Theorems \ref{th:transport-entropy} and \ref{th:transport-entropy-var}, respectively, together with Theorem \ref{th:entropy-lipwidth}, which compares entropy and Lipschitz width. We omit the log factors in the latter theorem in favor of a slightly diminished polynomial rate.

\end{proof}

These lower bounds have several practical implications:
\begin{enumerate}

  \item In order to achieve high order, $\sm+\sp$ must be large and therefore either the inflow boundary condition or the quantity of interest must be smooth. This also implies that high order methods for the PDE solution $u_\mu$ that are available in the literature, e.g. \cite{Welper2017,Welper2019}, cannot yield high order approximations of quantities of interest with full online/offline splitting, unless the functional is smooth, see Section \ref{sec:implications-for-transform-methods} for details.

  \item In the given setup, the regularity of the transport field $b_\mu$ is critical for the dimension dependence of the lower bounds. For only one $b$, the entropy (lower bound) is dimension independent. For arbitrary $C^{\sb}$ smooth $b$, we suffer the usual curse of dimensionality. For comparison, $u_\mu$ can be approximated with dimension independent bounds if the flow field is in Barron spaces \cite{LaakmannPetersen2020} or has compositional sparsity \cite{Dahmen2023}. Corresponding entropy benchmarks are left for future work. Also note that the situation is quite different form the elliptic case, where dimension independent orders are achieved for very high dimensional diffusion parameters.

\end{enumerate}

The results are proven in Section \ref{sec:transport:proof-main}. Section \ref{sec:transport-simple} discusses a very simple case, which highlights the relation of the entropy numbers to the known entropy of convolutions. Our subsequent proofs follow a similar argument. In Sections \ref{sec:transport:entropy-b}, \ref{sec:transport:entropy-b-var}, we provide conditions on the flow field $b$, which allow us to prove entropy bounds. In Section \ref{sec:transport:reference-domain}, we construct flow fields that satisfy the stated conditions and in Section \ref{sec:transport:proof-main} we prove the theorems above. We include some simple upper bounds in Section \ref{sec:transport:upper-bounds}.

\subsection{Non-Zero Right Hand Side}
\label{sec:rhs-entropy}

The main results consider transport problems $b_\mu \nabla u = f$ with right and side $f=0$. This choice $f=0$ is justified by the following lemma.

\begin{lemma}
  For problems
  \begin{align}
    \qi(\mu) & = \ell(u_\mu) = \int_{\Gamma_+} g_+ u_\mu, &
    b_\mu \nabla u_\mu & = f, &
    u_\mu|_{\Gamma_-} & = g_-
  \end{align}
  parametrized by
  \begin{equation*}
    \instances(b,g_+,f) := \left\{ g_- : \, g_- \in B_{W^{\sm,p}(\Gamma_-)} \right\}
  \end{equation*}
  we have
  \begin{equation*}
    \entropy_n(\Qi(\instances(b,g_+,f))) = \entropy_n(\Qi(\instances(b,g_+,0))).
  \end{equation*}
\end{lemma}

\begin{proof}

  Define $u_\mu^0$ and $u_\mu^f$ by 
  \begin{align*}
    b_\mu \nabla  u_\mu^0 & = 0 &
    b_\mu \nabla  u_\mu^f & = f \\
    u_\mu^0|_{\Gamma_-} & = g_- &
    u_\mu^f|_{\Gamma_-} & = 0,
  \end{align*}
  respectively. By linearity we have $u_\mu = u_\mu^0 + u_\mu^f$ and thus
  \begin{equation*}
    \qi(\mu) = \ell(u_\mu^0) + \ell(u_\mu^f).
  \end{equation*}
  Since the first summand is independent of the right hand side $f$ and the second is independent of the problem instance given by the inflow boundary $g_-$, all functions $\qi(\mu)$ are shifted by the single function $\mu \to \ell(u_\mu^0)$ between the two classes $\Qi(\instances(b, g_+, f))$ and $\Qi(\instances(b, g_+, 0))$, respectively. Clearly this leaves the entropy invariant. 

\end{proof}

\subsection{A Simple Case}
\label{sec:transport-simple}

Before we prove the main theorems, we discuss a simple problem class for which the entropy calculation can be reduced to a well-known case:
\begin{equation*} % \label{eq:simple-transport}
  \begin{aligned}
    \partial_t u_\mu + \mu \partial_x u_\mu & = 0 & &  \text{in } \real \times \real_+ \\
    u_\mu & = g_- & & \text{on } \{0\} \times \real,
  \end{aligned}
\end{equation*}
with parameter $\mu \in \param := [-1, 1]$ and quantities of interest
\begin{equation*} % \label{eq:def-output}
  \ell(u_\mu) = \int_\real u_\mu(x,1) g_+(x) \text{d}x
\end{equation*}
at some final time $t$, which we set to $t=1$ for simplicity. To ease the argument below, regarding spatial inflow boundary conditions, we choose the infinite domain $\real \times \real_+$. With solution
\begin{equation*}
  u_\mu(x,t) = g_-(x-\mu t)
\end{equation*}
we compute the quantity of interest explicitly as
\[
  \qi(\mu) 
  = \ell(u_\mu)
  = \int_\real g_-(x - \mu) g_+(x) \text{d}x.
\]
Setting $Rg_-(x) := g_-(-x)$, this simplifies to the convolution
\begin{equation*}
  \qi(\mu) = (Rg_- * g_+)(\mu).
\end{equation*}
Generally, we expect low entropy only if either $g_-$ or $g_+$ admit some smoothness. In the following, we use Sobolev smoothness, which allows us to construct proofs along standard lines \cite{LorentzGolitschekMakovoz1996} by constructing sufficiently many quantities of interest from disjoint bump functions.

\subsection{Entropy Bounds for Highly Parameter Dependent Flows: fixed \texorpdfstring{$b$}{b}}
\label{sec:transport:entropy-b}

In this section, we provide entropy bounds for flow fields $b_\mu$ that are strongly parameter dependent, characterized by the way they move small balls from inflow to outflow and back. To this end, we fix a parameter independent subset $\bo \subset \Gamma_+$ of the outflow boundary and define two maps that push points from inflow to outflow, and vice versa
\begin{equation} \label{eq:inflow-to-outflow}
  \begin{aligned}
    & \cf_\mu \colon \Gamma_- \to \Gamma_+ & 
    \cf_\mu(x) & = \Gamma_+ \cap \{X_{\mu}(t; x) : \, t \in \real\}  
    \\
    & \cb_\mu \colon \Gamma_+ \to \Gamma_- & 
    \cb_\mu(y) & = \Gamma_- \cap \{X_{\mu}(t; y) : \, t \in \real\}  
  \end{aligned}
\end{equation}
along characteristics
\begin{align*}
  \frac{d}{dt} X_\mu(t;x) & = b_\mu(X_\mu(t;x)), &
  X(0,x) & = x.
\end{align*}
We characterize the dependence of the flow field on the parameter by observing how $\cf_\mu$ and $\cb_\mu$ move small balls
\begin{align} \label{eq:balls-outflow}
  \bo[,h] & = \xi_+(B_h) \subset \Gamma_+, &
  \bo & := \bo[,1], &
  B_h = \{z \in \real^{d-1} : \, |z| \le h\}
\end{align}
defined by a parametrization $\xi_+$ of (part) of the outflow boundary $\Gamma_+$. First, we ensure that $\cf_\mu$ and $\cb_\mu$ are well defined.
\begin{assumption} \label{assumption:backward-forward}
  The two maps $\cf_\mu$ and $\cb_\mu$ are well-defined, single-valued and satisfy
  \begin{equation} \label{eq:assumption:forward-backward}
    \begin{aligned} 
      \cf_\mu \circ \cb_\mu(y) & = y, &
      y & \in \bo.
      \\
      \cb_\mu \circ \cf_\mu(x) & = x, &
      x & \in \cb(\bo).
    \end{aligned}
  \end{equation}
\end{assumption}

Informally, condition \eqref{eq:assumption:forward-backward} requires that for every parameter $\mu$, the characteristics through the outflow $\bo$ can be traced backward to some part of the inflow boundary $\Gamma_-$. We ensure the richness of the transport directions as follows. Transporting a small ball $\bo[,h]$ backwards with different parameters, we obtain disjoint supports.

\begin{assumption} \label{assumption:backward-sets}
  For every $0 < h \le 1$ there are $n \sim h^{-\diml}$, $\diml \le d-1$ parameters $\mu_1, \dots, \mu_n$ such that $\cb_{\mu_i}(\bo[,h])$ are disjoint. See Figure \ref{fig:backward-sets}.
\end{assumption}

\begin{figure}[htb]
  \hfill
  \begin{subfigure}[b]{0.4\textwidth}
    \centering
    \begin{tikzpicture}[scale=0.75]

      \draw [thick]
            (0,0) 
                  .. controls (0,-1) and (3.5,-1) .. 
            (4,0) node[midway,circle,fill=blue] (m1) {} 
                  .. controls (4.5, 1) and (4.5, 1) .. 
            (4,2) 
                  .. controls (3.5, 3) and (0, 1) .. 
            (0,0) node[midway,circle,fill=blue] (m2) {} ;

      \draw [thick,red]    
            (4,0) .. controls (4.5, 1) and (4.5, 1) .. 
            (4,2)
            node[pos=0.8,right] {$\bo$}
            node[pos=0.4,circle,fill=blue] (B) {}
            node[pos=0.4,right,blue] {$\bo[,h]$};

      \draw[thick,blue,->] (B) to[out=190,in=40] node[midway,below] {$\mu_1$} (m1);
      \draw[thick,blue,->] (B) to[out=170,in=-40] node[midway,below] {$\mu_2$} (m2);

    \end{tikzpicture}
    \caption{Assumption \ref{assumption:backward-sets}.}
    \label{fig:backward-sets}
  \end{subfigure}
  \hfill
  \begin{subfigure}[b]{0.4\textwidth}
    \centering
    \begin{tikzpicture}[scale=0.75]

      \draw [thick]
            (0,0) 
                  .. controls (0,-1) and (3.5,-1) .. 
            (4,0) node[midway,circle,fill=blue] (m1) {} 
                  .. controls (4.5, 1) and (4.5, 1) .. 
            (4,2) 
                  .. controls (3.5, 3) and (0, 1) .. 
            (0,0) node[midway,circle,fill=blue] (m2) {} ;

      \draw [thick,red]    
            (4,0) .. controls (4.5, 1) and (4.5, 1) .. 
            (4,2)
            node[pos=0.8,right] {$\bo$}
            node[pos=0.4,circle,fill=blue] (B) {}
            node[pos=0.4,right,blue] {$\bo[,h]$}
            node[pos=0.3,circle,fill=brown] (f1) {}
            node[pos=0.9,circle,fill=brown] (f2) {};

      \draw[thick,blue,->] (B) to[out=190,in=40] node[pos=0.7,below] {$\mu_1$} (m1) node[below right] {$\in M_1$};
      \draw[thick,blue,->] (B) to[out=170,in=-40] node[pos=0.7,below] {$\mu_2$} (m2) node[above] {$\not\in M_2$};
      \draw[thick,brown,->] (m1) to[out=40,in=190] node[pos=0.7,below] {$\mu$} (f1);
      \draw[thick,brown,->] (m2) to[out=-40,in=170] node[pos=0.7,below] {$\mu$} (f2);

    \end{tikzpicture}
    \caption{Assumption \ref{assumption:forward-backward-sets} for $\mu \approx \mu_1$.}
    \label{fig:forward-backward-sets}
  \end{subfigure}
  \hfill~

  \caption{Illustration of assumptions.}
  
\end{figure}
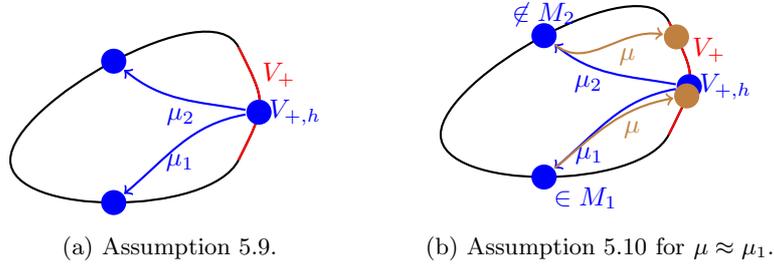

In addition, going backward with parameter $\mu_i$ and forward with $\mu$ will generally end up in different supports, unless $\mu$ is close to $\mu_i$:
\begin{assumption} \label{assumption:forward-backward-sets}
  For all $i=1, \dots, n$, the sets
  \[
    M_i := \{ \mu \in \param : \cf_\mu \circ \cb_{\mu_i}(\bo[,h]) \cap \bo[,h] \ne \emptyset \}
  \]
  are pairwise disjoint. See Figure \ref{fig:forward-backward-sets}.
\end{assumption}

Finally, we assume that the flow field is smooth in the following sense.
\begin{assumption} \label{assumption:smooth}
  For $\sm,\sp>0$, $1 \le p < \infty$, assume that the linear operator 
  \begin{align}
    & \bar{\cb}_\mu \colon W_0^{\sm,p}(\bo) \to W^{\sm,p}(\Gamma_-), & 
    \bar{\cb}_\mu g_+ & = g_+ \circ \cb_\mu
  \end{align}
 is bounded, uniformly in $\mu$.
\end{assumption}

Flow fields that satisfy all assumptions are constructed in Section \ref{sec:transport:reference-domain}. Given that we already have such a flow field, we obtain the following entropy bounds.

\begin{proposition}
  \label{prop:transport-entropy}

  Let $b$ be a flow field and $\instances(b, g_+)$ be the problem class defined in \eqref{eq:transport-class} and \eqref{eq:transport-class-parametrization-inflow}. Let Assumptions \ref{assumption:backward-forward}, \ref{assumption:backward-sets}, \ref{assumption:forward-backward-sets} and \ref{assumption:smooth} be satisfied with $\diml = d-1$. Then there is a outflow quantity of interest $g_+ \in B_{W^{\sp,\infty}(\partial \Omega})$ with
  \begin{equation*}
    \entropy_n(\Qi(\instances(b, g_+)) \gtrsim n^{-\frac{\sm+\sp+ (d-1)}{d-1}}.
  \end{equation*}

\end{proposition}

The proof is based on bump functions constructed in the following lemma.

\begin{lemma} \label{lemma:bump-functions}

  Let Assumptions \ref{assumption:backward-forward}, \ref{assumption:backward-sets}, and \ref{assumption:smooth} be satisfied. Then, there are functions $g_{-,i}$, $i=1, \dots, n$ and $g_+$ such that
  \begin{align*}
    \left\| \sum_{i=1}^n \theta_i g_{-,i} \right\|_{W^{\sm,p}(\Gamma_-)} & \le 1, &
    \supp(g_{-,i}) & = \supp(g_+ \circ \cf_{\mu_i}),
    \\
    \left\| g_+ \right\|_{W^{\sp,\infty}(\Gamma_+)} & \le 1, &
    \supp(g_+) & \subset \bo[,h], 
  \end{align*}
  and
  \begin{align*}
    \ell(g_{-,i} \circ \cb_{\mu_i})
    & \sim h^{\sm+\sp+ (d-1)}, &
    \ell(v)
    & := \int_{\Gamma_+} g_+ v
  \end{align*}
  for $\sm,\sp>0$, $1 \le p < \infty$, $n \sim h^{-\diml}$, $\diml \le d-1$ and any $\theta_i \in \{0, 1\}$. In addition
  \begin{align*}
    \supp\left( \mu \to \ell(g_{-,i} \circ \cb_{\mu}) \right) & \subset M_i, &
    M_i & := \{\mu: \bo[,h] \cap \cf_\mu \circ \cb_{\mu_i}(\bo[,h]) \ne \emptyset\}
  \end{align*}
\end{lemma}

\begin{proof}

All $g_{-,i}$ and $g_+$ are constructed from a bump-function $\psi: \real^{d-1} \to \real$ with $\psi \in C^\infty(\real^{d-1})$, $0 \le \psi \le 1$, $\operatorname{supp}(\psi) \subset B_1$, and $\|\psi\|_1 \sim \|\psi\|_2 \sim 1$. For $i = 1, \dots, n$, we first define the functions
\begin{align*}
  \psi_h(z) & := \psi(z/h), \\
  \psi_h^+(y) & := (\psi_h \circ \xi_+^{-1})(y), \\
  \psi_{h,i}^-(x) & := (\psi_h^+ \circ \cf_{\mu_i})(x),
\end{align*}
where $\xi_+$ is the parametrization of the outflow boundary defined in \eqref{eq:balls-outflow}. The regularity Assumption \ref{assumption:smooth} and standard scaling arguments yield
\begin{align*} % \label{eq:bump-function-smoothness}
  \|\psi_{h,i}^-\|_{W^{\sm,p}(\Gamma_-)} & \lesssim h^{-\sm+(d-1)/p}, &
  \|\psi_h^+\|_{W^{\sp,\infty}(\Gamma_+)} & \lesssim h^{-\sp}
\end{align*}
uniformly in $i$. Define
\begin{align}
  g_{-,i} & := c_{\sm} h^{\sm} \psi_{h,i}^-, &
  g_+ & := c_{\sp} h^{\sp} \psi_h^+
\end{align}
for some constants $c_{\sm}$ and $c_{\sp}$. Then, we have
\begin{equation*}
  \left\| g_+ \right\|_{W^{\sp,\infty}(\Gamma_+)} 
  = c_{\sp} h^{\sp} \left\| \psi_+ \right\|_{W^{\sp,\infty}(\Gamma_+)} 
  \le 1
\end{equation*}
for $c_{\sp}$ sufficiently small and
\begin{multline*}
  \left\| \sum_{i=1}^n \theta_i g_{-,i} \right\|_{W^{\sm,p}(\Gamma_-)}^p 
  \le c_{\sm}^p h^{ps} \sum_{i=1}^n \|\psi_{h,i}^-\big\|_{W^{\sm,p}(\Gamma_-)}^p 
  \\
  \le c c_{\sm}^p h^{ps} \sum_{i=1}^n h^{-ps + (d-1)} 
  \le c c_{\sm}^p n h^{d-1}
  \le 1
\end{multline*}
for $c_{\sm}$ sufficiently small, where we have used that all $\psi_{h,i}^-$ have disjoint support and $n \sim h^{-\diml} \le h^{-(d-1)}$ by Assumption \ref{assumption:backward-sets}. By construction, we have
\begin{align*}
  \supp(g_{-,i}) & = \supp(g_+ \circ \cf_{\mu_i}), &
  \supp(g_+) & \subset \bo[,h], 
\end{align*}
Since $\cf_{\mu_i} \circ \cb_{\mu_i} = \operatorname{Id}$
\begin{multline*}
  \ell(g_{-,i} \circ \cb_{\mu_i})
  = \int_{\Gamma_+} (g_{-,i} \circ \cb_{\mu_i}) g_+ 
  = c_{\sm} c_{\sp} h^{\sm+\sp} \int_{\Gamma_+} (\psi_h^+ \circ \cf_{\mu_i} \circ \cb_{\mu_i}) \psi_h^+   
  \\
  = c_{\sm} c_{\sp} h^{\sm+\sp} \int_{\Gamma_+} (\psi_h^+)^2
  \sim h^{\sm+\sp+ (d-1)}.
\end{multline*}
Finally, we compute the support of $\mu \to \ell(g_{-,i} \circ \cb_{\mu}) = \mu \to \int_{\Gamma_+} g_+ (g_{-,i} \circ \cb_{\mu})$. This function is non-zero only if $\supp(g_+) \cap \supp(g_{-,i} \circ \cb_{\mu}) \ne 0$. Since $g_{-,i} \circ \cb_{\mu}(y) \ne 0$ implies $\cf_{\mu_i} \circ \cb_\mu(y) \in \supp(\psi_+) \subset \bo[,h]$ equivalent to $y \in \cb_\mu^{-1} \circ \cf_{\mu_i}^{-1} (\bo[,h]) = \cf_\mu \circ \cb_{\mu_i}(\bo[,h])$, we have
\begin{equation*}
  \{\mu : \, \supp(g_+) \cap \supp(g_{-,i} \circ \cb_\mu) \ne \emptyset\}
  \subset \{\mu: \bo[,h] \cap \cf_\mu \circ \cb_{\mu_i}(\bo[,h]) \ne \emptyset\}
  = M_i,
\end{equation*}
which implies the statement in the lemma.

\end{proof}

The following result from \cite{LorentzGolitschekMakovoz1996} provides the entropy from bump functions.

\begin{lemma}[{\cite[Chapter 15, Proposition 2.3]{LorentzGolitschekMakovoz1996}}]
  \label{lemma:entropy-lower-bound-construction}
  Assume there exists functions $\qi_i \in L_\infty(\param)$, $i=1, \dots, n$ with disjoint supports such that for $\epsilon > 0$ the inequality
  \begin{equation*}
    \|\qi_i\|_\infty > \epsilon
  \end{equation*}
  holds and 
  \begin{equation*} % \label{eq:lower-bound-construction-2}
    \sum_{i=1}^n \theta_i \qi_i \in K
  \end{equation*}
  for all $\theta \in \{0,1\}^n$. Then
  \begin{equation*}
    \entropy_{n-1}(K) > \epsilon.
  \end{equation*}
\end{lemma}

\begin{proof}
  The statement in the reference \cite[Chapter 15, Proposition 2.3]{LorentzGolitschekMakovoz1996} is for capacity instead of entropy, but the short proof is almost unchanged. It is a special case of Lemma \ref{lemma:entropy-lower-bound-count}, below.
\end{proof}

\begin{proof}[Proof of Proposition \ref{prop:transport-entropy}]

Let $g_{-,i}$ and $g_+$ be defined as in Lemma \ref{lemma:bump-functions}. By the provided regularity of these functions, the problems
\begin{align*}
  \qi_i(\mu) & = \ell(u_\mu) = \int_{\Gamma_+} g_+ u_\mu, &
  b_\mu \nabla u_\mu & = 0, &
  u_\mu|_{\Gamma_-} & = g_{-,i}
\end{align*}
are in our problem class. We conclude the proof by Lemma \ref{lemma:entropy-lower-bound-construction} by showing that $\qi_i$ satisfy its assumptions. To this end, note that by the method of characteristics the solutions of the transport equation are given by $u_\mu(y) = g_{-,i} \circ \cb_\mu(y)$ for all $y \in \Gamma_+$ in the outflow boundary. Hence $\qi_i(\mu) = \ell(g_{-,i} \circ \cb_\mu)$ and by Lemma \ref{lemma:bump-functions}
\begin{equation*}
  \supp(\qi_i) 
  \subset \{\mu: \bo[,h] \cap \cf_\mu \circ \cb_{\mu_i}(\bo[,h]) \ne \emptyset \}
  = M_i,
\end{equation*}
which are disjoint by Assumption \ref{assumption:forward-backward-sets}. Next, using again the characteristic solution and Lemma \ref{lemma:bump-functions}, we have
\begin{equation*}
  \|\qi_i\|_Z 
  = \sup_{\mu \in \param} |\qi_i(\mu)| 
  \ge |\qi_i(\mu_i)| 
  = \ell(g_{-,i} \circ \cb_{\mu_i})
  \sim h^{\sm+\sp+(d-1)}  
  =: \epsilon.
\end{equation*}
Finally by the method of characteristics, for any $\theta_i \in \{0,1\}$, the sum
\begin{equation*}
  \sum_{i=1}^n \theta_i \qi_i
  = \sum_{i=1}^n \theta_i \ell(g_{-,i} \circ \cb_\mu)
  = \ell \left( \left( \sum_{i=1}^n \theta_i g_{-,i} \right) \circ \cb_\mu \right)
\end{equation*}
is the output functional for a transport equation with initial value $\sum_{i=1}^n \theta_i g_{-,i}$, which is contained in the problem class by Lemma \ref{lemma:bump-functions}. Thus, $\qi_i$ satisfy all assumptions of Lemma \ref{lemma:entropy-lower-bound-construction} and we obtain
\begin{equation*}
  \entropy_n(\Qi(\instances(b, g_+))) \gtrsim \epsilon = h^{\sm+\sp+(d-1)}.
\end{equation*}
Thus, the result follows from $n \sim h^{-(d-1)}$ or equivalently $h \sim n^{-\frac{1}{d-1}}$.

\end{proof}

\subsection{Entropy Bounds for Highly Parameter Dependent Flows: variable \texorpdfstring{$b$}{b}}
\label{sec:transport:entropy-b-var}

In this section, we provide entropy bounds that demonstrate the (parameter) curse of dimensionality for a class of transport problems with variable flow fields constructed from one strongly parameter dependent reference field. We start from a parameter set $\param \subset \real^{d-2+D}$ that can be split into two components $\mu = (\mua, \mub) \in \parama \times \paramb \subset \real^{d-2} \times \real^D$. The first one is used as in the last section. The second is used to show dimension dependence of the entropy.

For now, let us assume that there is a strongly parameter dependent flow field $b_{\mua, \eta}$ and defer its construction to Section \ref{sec:transport:reference-domain} below. Unlike $\mub \in \real^D$, the parameter $\eta \in \real$ is of dimension one and serves as a ``switch'' between different behaviours of $b$. The switch is triggered locally by some function $\eta(\mub)$, resulting in the parametric flow field $b(x, \mua, \eta(\mub))$ depending on the full $\real^{d-2+D}$ dimensional parameter $(\mua, \mub)$. In the first switch position $\eta = \delta := 5h$, the flow field has the same properties as in the previous section. In the alternate position $0 = \eta = \eta(\mub)$, we ensure that $\qi(\mua, \mub) = 0$ vanishes by the following assumption. This allows us to control the support of $\qi(\mua, \mub)$ similar to Lemma \ref{lemma:entropy-lower-bound-construction}.

\begin{assumption} \label{assumption:extra-dim-independent}
  For all $i=1, \dots, n$, the sets
  \[
    \{ \mu \in \param : \cf_{\mu, 0} \circ \cb_{\mu_i, \delta}(\bo[,h]) \cap \bo[,h] \ne \emptyset \} = \emptyset
  \]
  are empty. See Figure \ref{fig:assumption:extra-dim-independent}.
\end{assumption}

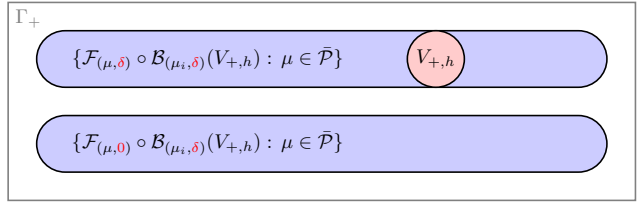
\begin{figure}[htb]

  \begin{center}
    \scalebox{0.75}{
      \begin{tikzpicture}

        \draw[thick,gray] (-0.5, 1.5) node[below right] {$\Gamma_+$} rectangle (10.5,-2);
        \draw[thick,rounded corners=0.5cm, fill=blue!20] (0,0) rectangle (10,1);
        \draw[thick,rounded corners=0.5cm, fill=blue!20] (0,-1.5) rectangle (10,-0.5);
        \draw[thick,rounded corners=0.5cm, fill=red!20] (6.5,0) rectangle (7.5,1);

        \draw (7,0.5) node {$\bo[,h]$};
        \draw (3,0.5) node {$\{\cf_{(\mu,\fignote{\delta})} \circ \cb_{(\mu_i,\fignote{\delta})}(\bo[,h]) : \, \mu \in \parama\}$};
        \draw (3,-1) node {$\{\cf_{(\mu,\fignote{0})} \circ \cb_{(\mu_i,\fignote{\delta})}(\bo[,h]) : \, \mu \in \parama\}$};
        
      \end{tikzpicture}
    }
  \end{center}
  \caption{Illustration of Assumption \ref{assumption:extra-dim-independent}.}
  \label{fig:assumption:extra-dim-independent}
  
\end{figure}
The assumption is enabled by choosing $\mua$ of dimension $d-2$, one less than in the previous section for fixed $b$. This allows us to move the balls $\bo[,h]$ in a unused dimension when changing $\eta$, which ensures that the backward and then forward characteristics for different $\eta$ always miss their origin. A detailed construction of flow fields that satisfy Assumption \ref{assumption:extra-dim-independent} is given in Section \ref{sec:transport:reference-domain}. For now, we assume it is satisfied and prove the following entropy bound.

\begin{remark}
  Since our flow field $b_{\mua, \eta}$ has two indices all quantities from the last section have so as well. In particular, we use the notation
  \begin{align*}
    \cf_{(\mua, \eta)} & = \cf_{\mua, \eta}, &
    \cb_{(\mua, \eta)} & = \cb_{\mua, \eta}. &
  \end{align*}
\end{remark}

\begin{proposition} \label{prop:transport-entropy-var}

  Let $\instances$ be the problem class defined by \eqref{eq:transport-class}, \eqref{eq:transport-class-parametrization} with variable flow field. Assume there is a parametric transport field $b_{\mua, \eta}$ with $(\mua, \eta) \in \parama \times [-1,1]$ so that
  \begin{enumerate}
    \item For frozen second parameter $\eta=\delta=5h$, $b_{\mua, 5h}$ satisfies Assumptions \ref{assumption:backward-forward}, \ref{assumption:backward-sets}, \ref{assumption:forward-backward-sets} and \ref{assumption:smooth} with parameters $(\mua_i, 5h)$, $i = 1, \dots, n$, $n \sim h^{-\diml}$, $\diml = d-2$.
    \item $b_{\mua, \eta}$ satisfies Assumption \ref{assumption:extra-dim-independent}.
    \item $\|b\|_{C^{\sb}(\Omega \times \parama \times [0,5h])} \lesssim 1$.
  \end{enumerate}
  Then
  \begin{equation*}
    \entropy_n(\Qi(\instances)) \gtrsim n^{-\frac{\sm+\sp+ (d-1)}{\max\{D/\sb, d-2\}}}.
  \end{equation*}
  
\end{proposition}

We show the entropy bound by constructing the functions in the following lemma.

\begin{lemma}
  \label{lemma:entropy-lower-bound-count}
  For domain $\param$, let $K \subset L_\infty(\param)$. Assume there are elements $\qi_\theta \in K$ with $\theta$ in some finite index set $\Theta$ so that for any $\theta \ne \vartheta \in \Theta$ there is a $\mu_{\theta\vartheta} \in \param$ with
  \begin{equation*}
    |\qi_\theta(\mu_{\theta\vartheta}) - \qi_\vartheta(\mu_{\theta\vartheta})| > \epsilon.
  \end{equation*}
  Then
  \begin{equation*}
    \entropy_{\lfloor \log_2 (|\Theta|-1) \rfloor}(K) > \frac{\epsilon}{2}.
  \end{equation*}
\end{lemma}

\begin{proof}

First note that
\begin{equation*}
  \|\qi_\theta - \qi_\vartheta\|_\infty
  \ge |\qi_\theta(\mu_{\theta\vartheta}) - \qi_\vartheta(\mu_{\theta\vartheta})| 
  \ge \epsilon.
\end{equation*}
Now assume that the assertion is false and we have $\entropy_{\lfloor \log_2 (|\Theta|-1) \rfloor}(K) \le \frac{\epsilon}{2}$. Then there is a $\epsilon/2$-cover with at most $2^{\log_2(|\Theta|-1)} = |\Theta|-1$ balls with some centers $\qi_i$. Since this is one smaller than the size of $\Theta$, there must be one ball, centered at some $\bar{\qi}$, containing two different $\qi_\theta$ and $\qi_\vartheta$ for $\theta \ne \vartheta$. Thus
\begin{equation*}
  \|\qi_\theta - \qi_\vartheta\|_\infty
  \le \|\qi_\theta - \bar{\qi}\|_\infty + \|\bar{\qi} - \qi_\vartheta\|_\infty
  \le \frac{\epsilon}{2} + \frac{\epsilon}{2}
  = \epsilon.
\end{equation*}
Since this contradicts $\|\qi_\theta - \qi_\vartheta\|_\infty > \epsilon$, we must have $\entropy_{\lfloor \log_2 (|\Theta|-1) \rfloor}(K) > \frac{\epsilon}{2}$ and the lemma follows.

\end{proof}

\begin{proof} [Proof of Proposition \ref{prop:transport-entropy-var}]

\begin{itemize}

  \item \emph{Construction of quantity of interest:} The result follows from Lemma \ref{lemma:entropy-lower-bound-count} given that we can construct the necessary quantities of interest. To this end, define the inflow boundary condition $g_{-,i}$ and outflow quantity of interest $g_+$ by Lemma \ref{lemma:bump-functions} with parameters $(\mua_i, 5h) = (\mua_i, \delta)$ and $\diml = d-2$. The lemma also shows that for all $\theta \in \{0, 1\}^n$ the linear combinations
  \[
    g_{-,\theta} := \sum_{i=1}^n \theta_i g_{-,i}
  \]
  as well as $g_+$ are smooth.

  In order to define multiple transport fields, let $\phi \colon \real^D \to \real$ be a bump function
  \begin{align*}
    \phi(0) & = 1, &
    \supp(\phi) & \subset B_1, &
    \phi & \in C^\infty(\paramb), &
    \|\phi\|_{L_1(\paramb)} & \sim 1
  \end{align*}
  and $\phi_k$, $k = 1, \dots, K$ be scaled and translated bump functions with
  \begin{align*}
    \phi_k(\mub) & = \delta \phi(h^{-\frac{1}{\sb}} (\mub - \mub_k)), &
    \delta & = 5h, &
    \|\phi_k\|_{W^{\sb, \infty}(\paramb)} & \lesssim 1, &
    K & \lesssim h^{-D/\sb}
  \end{align*}
  for some grid points $\mub_k$ at distances $|\mub_k - \mub_{\bar{k}}| \ge 5 h^{\frac{1}{\sb}}$ so that all $\phi_k$ have disjoint support. For any $\vartheta \in \{0,1\}^K$ define the transport fields
  \begin{align*}
    b_{\mua, \mub}^\vartheta & := b_{\mua, \vartheta(\mub)}, &
    \vartheta(\mub) = \sum_{k=1}^K \vartheta_k \phi_k(\mub).
  \end{align*}
  Unlike the given transport filed $b_{\mua, \eta}$ with $(\mua, \eta) \in \real^{d-2+1}$, the new fields $b_{\mua,\mub}^\vartheta$ have parameter dimension $(\mua, \mub) \in \real^{d-2+D}$ of the problem class. Since all $\phi_k$ have disjoint support,
  \begin{equation*}
    |D^{\sb} \vartheta(\mub)| 
    = |D^{\sb} \phi_k(\mub)| 
    = \left| 5h \left(h^{-\frac{1}{\sb}}\right)^{\sb} (D^{\sb} \phi)\left( h^{-\frac{1}{\sb}}(\mub-\mub_k \right) \right| 
    \lesssim \|D^{\sb} \phi\|_\infty
    \lesssim 1
  \end{equation*}
  for some $k \in \{1, \dots, K\}$ or $D^{\sb} \vartheta(\mub) = 0$ if $\mub$ is not in the support. Thus by the chain rule all flow fields satisfy $\|b^\vartheta\|_{C^{\sb}(\Omega \times \parama \times \paramb)} \lesssim 1$. In conclusion, all quantities of interest
  \begin{align*}
    \qi_{\theta, \vartheta}(\mua,\mub) & = \ell(u_{\mua,\mub}) = \int_{\Gamma_+} g_+ u_{\mua,\mub}, &
    b_{\mua,\mub}^\vartheta \nabla u_{\mua,\mub} & = 0, &
    u_{\mua,\mub}|_{\Gamma_-} & = g_{-,\theta}
  \end{align*}
  are in our problem class.

  \item \emph{Evaluation of quantity of interest at grid points:} In order to show the assumptions of Lemma \ref{lemma:entropy-lower-bound-count}, we first compute $\qi_{\theta, \vartheta}(\mua_i, \mub_k)$. Given the pair $(i,k)$, first assume that $\vartheta_k = 1$ so that $\vartheta(\mub_k) = \sum_{\bar{k}=1}^K \vartheta_{\bar{k}} \phi_{\bar{k}}(\mub_k) = \phi_k(\mub_k) = \delta$. Then the characteristic solution is given by $u_{\mua_i, \mub_k} = g_{-,\theta} \circ \cb_{\mua_i,\delta}$ and the quantity of interest by
  \begin{equation*}
    \qi_{\theta,\vartheta}(\mua_i, \mub_k) 
    = \ell(u_{\mua_i, \mub_k} )
    = \ell \left( g_{-,\theta} \circ \cb_{\mua_i,\delta} \right)
    = \sum_{\bar{\imath}=1}^n \theta_{\bar{\imath}} \ell \left( g_{-,\bar{\imath}} \circ \cb_{\mua_i,\delta} \right).
  \end{equation*}
  By Lemma \ref{lemma:bump-functions}, with parameter $\mu_i = (\mua_i, \delta)$ and Assumption \ref{assumption:forward-backward-sets} the summands $\ell \left( g_{-,\bar{\imath}} \circ \cb_{\mua_i,\delta} \right)$ are zero if $\bar{\imath} \ne i$ and $\sim h^{\sm+\sp+ (d-1)}$ if $\bar{\imath} = i$. Hence
  \begin{equation*}
    \qi_{\theta,\vartheta}(\mua_i, \mub_k) 
    \sim \theta_i h^{\sm+\sp+ (d-1)}.
  \end{equation*}
  On the other hand, if $\vartheta_k = 0$, we have $\eta(\mub_k) = 0$ and thus
  \begin{equation*}
    \qi_{\theta,\vartheta}(\mua_i, \mub_k) 
    = \ell(u_{\mua_i, \mub_k} )
    = \ell \left( g_{-,\theta} \circ \cb_{\mua_i,0} \right)
    = \sum_{\bar{\imath}=1}^n \theta_{\bar{\imath}} \ell \left( g_{-,\bar{\imath}} \circ \cb_{\mua_i,0} \right) = 0,
  \end{equation*}
  where we have used that $g_{-,\bar{\imath}}$ is constructed from parameters $(\mua_{\bar{\imath}}, \delta)$ and by Lemma \ref{lemma:entropy-lower-bound-count} we have $\supp \left( (\mua,0) \to \ell(g_{-,\bar{\imath}} \circ \cb_{\mua,0}) \right) \subset \bo[,h] \cap \cf_{\mua,0} \circ \cb_{\mua_{\bar{\imath}}, \delta}(\bo[,h])$, which is empty by Assumption \ref{assumption:extra-dim-independent}. In summary
  \begin{equation*}
    \qi_{\theta,\vartheta}(\mua_i, \mub_k) 
    \sim \left\{ \begin{array}{ll}
      h^{\sm+\sp+ (d-1)} & \theta_i \ne 0\text{ and }\vartheta_k \ne 0 \\
      0 & \text{else}.
    \end{array} \right.
  \end{equation*}
  
  \item \emph{Assumptions of Lemma \ref{lemma:entropy-lower-bound-count}:} We show that the assumptions of Lemma \ref{lemma:entropy-lower-bound-count} are satisfied for the parameters $\mu_{\theta\vartheta} = (\mua_i, \mub_k)$ for $i$ and $k$ chosen as follows.
  
  Let $(\theta, \vartheta) \ne (\bar{\theta}, \bar{\vartheta})$ be two parameters with $\theta, \vartheta, \bar{\theta}, \bar{\vartheta}$ non-zero. If $\theta \ne \bar{\theta}$, there is a $i$ with $\theta_i = 1$, $\bar{\theta}_i = 0$ and a $k$ with $\vartheta_k = 1$. Likewise, if $\vartheta \ne \bar{\vartheta}$, there is a $k$ with $\vartheta_k = 1$, $\bar{\vartheta}_k = 0$ and a $i$ with $\theta_i = 1$. Therefore, there is an index pair $(i,k)$ with $\theta_i = 1$, $\vartheta_k=1$ and either $\bar{\theta}_i =0$ or $\bar{\vartheta}_k = 0$. Thus
  \begin{align*}
    \qi_{\theta,\vartheta}(\mua_i, \mub_k) 
    & \sim h^{\sm+\sp+ (d-1)} =: \epsilon, &
    \qi_{\bar{\theta},\bar{\vartheta}}(\mua_i, \mub_k) 
    & = 0.
  \end{align*}
  and the assumption of Lemma \ref{lemma:entropy-lower-bound-count} is satisfied.
  
  \item \emph{Conclusion:} Since there are
  \begin{equation*}
    \sim 2^{h^{-D/\sb}} 2^{h^{-(d-2)}}
    = 2^{h^{-D/\sb} + h^{-(d-2)}}
    \ge 2^{h^{-\max\{D/\sb, d-2\}}}
  \end{equation*}
  pairs $(\theta, \vartheta)$ in total, Lemma \ref{lemma:entropy-lower-bound-count} implies
  \begin{equation*}
    \entropy_{\lfloor h^{-\max\{D/\sb, d-2\}} - 1 \rfloor} \gtrsim h^{\sm+\sp+ (d-1)}.
  \end{equation*}
  Defining $n := \lfloor h^{-\max\{D/\sb, d-2\}} - 1 \rfloor$ and eliminating $h$ shows the lemma.

\end{itemize}

\end{proof}

\subsection{Reference Domains}
\label{sec:transport:reference-domain}

We have constructed lower entropy bounds given that the flow fields have sufficient parameter dependence, measured by its transport of small balls between inflow and outflow boundary. In this section, we construct some flow fields that satisfy all requirements. We do so first on a reference domain $[-1,1]^d$ that is mapped to a subset of $\Omega$ by the map $\rdom$ in Assumption \ref{assumption:rdom:reference-domain}.

The following parameter dependent version of $\rdom$ is used to define characteristics.

\begin{lemma} \label{lemma:rdom:one-to-one}
  Let $\param, W \subset \frac{1}{2} [-1,1]^{d-1}$. Then
  \begin{equation*}
    \rdom_\mu(t,w) = \rdom(t, (1-t)(w+\mu) + tw)
  \end{equation*}
  has range contained in $\Omega$ and is one-to-one.
  
\end{lemma}

\begin{proof}
  By definition all pairs $(t, (1-t)(w+\mu) + tw)$ are contained in the domain $[-1,1] \times [-1,1]^{d-1}$ of $\rdom$ and thus their image is contained in $\Omega$. To show uniqueness, assume
  \begin{equation*}
    \rdom(t, (1-t)(w+\mu) + tw)
    = \rdom(\bar{t}, (1-\bar{t})(\bar{w}+\mu) + \bar{t}\bar{w})
  \end{equation*}
  Since $\rdom$ is one-to-one, we have
  \begin{align*}
    t & = \bar{t}, &
    (1-t)(w+\mu) + tw & = (1-\bar{t})(\bar{w}+\mu) + \bar{t}\bar{w}.
  \end{align*}
  With $t=\bar{t}$, the second identity implies
  \begin{align*}
    (1-t)(w+\mu) + tw & = (1-t)(\bar{w}+\mu) + t\bar{w} \\
    \Leftrightarrow (1-t)w + t w & = (1-t)\bar{w} + t \bar{w} \\
    \Leftrightarrow w & = \bar{w}.
  \end{align*}
\end{proof}

For abbreviation, we denote the restrictions to inflow and outflow boundary by
\begin{align*}
  \xi_{-,\mu} & := \rdom_\mu(0,\cdot), &
  \xi_{+,\mu} & := \rdom_\mu(1,\cdot).
\end{align*}
The next lemma defines a flow field with corresponding characteristics.

\begin{lemma} \label{lemma:rdom:characteristics}
  For $(t,x) \in \ran(\rdom) \cap \Gamma_-$, define
  \begin{align*}
    X_\mu(t,x) & := \rdom_\mu(t, \xi_{-,\mu}^{-1}(x)), &
    b_\mu(t,x) & = \left[ (\partial_t \rdom_\mu) \circ \rdom_\mu^{-1} \right] (t,x).
  \end{align*}
  Then $X_\mu$ satisfies the characteristic equation
  \begin{align*}
    \frac{d}{dt} X_\mu(t,x) & = b_\mu(X_\mu(t,x)), &
    X_\mu(0,x) & = x
  \end{align*}
\end{lemma}

\begin{proof}
  
The initial value $X_\mu(0,x) = x$ follows directly from the definition of $\xi_{-,\mu}$. Abbreviating $w = \xi_{-,\mu}^{-1}(x)$, the differential equation follows from
\begin{align*}
  \frac{d}{dt} X_\mu(t,x)
  & = (\partial_t \rdom_\mu)(t, w)
  \\
  & = (\partial_t \rdom_\mu) \circ \rdom_\mu^{-1} \circ \rdom_\mu (t, w)
  \\
  & = (\partial_t \rdom_\mu) \circ \rdom_\mu^{-1} (X_\mu (t, x))
  \\
  & = b_\mu(X_\mu(t,x)).
\end{align*}

\end{proof}

The characteristic maps from inflow to outflow and back are explicitly given by $\xi_{-,\mu}$ and $\xi_{+,\mu}$.

\begin{lemma} \label{lemma:rdom:inflow-to-outflow}
  Let $\cf_\mu$ and $\cb_\mu$ be the characteristic maps between inflow and outflow boundary defined in \eqref{eq:inflow-to-outflow} for the characteristics defined in Lemma \ref{lemma:rdom:characteristics}. Then
  \begin{align*}
    \cf_\mu & = \xi_{+,\mu} \circ \xi_{-,\mu}^{-1}, & 
    \cb_\mu & = \xi_{-,\mu} \circ \xi_{+,\mu}^{-1}.
  \end{align*}
\end{lemma}

\begin{proof}
  For $y \in \ran(\xi_{+,\mu}) \subset \bo \subset \Gamma_+$, we have
  \begin{equation*}
    y 
    = \rdom_\mu(1, \xi_{+,\mu}^{-1}(y))
    = \rdom_\mu(1, \xi_{-,\mu}^{-1} \circ \xi_{-,\mu} \circ \xi_{+,\mu}^{-1}(y))
    = X_\mu(1, \xi_{-,\mu} \circ \xi_{+,\mu}^{-1}(y))
  \end{equation*}
  Since the latter is the characteristic originating at $\xi_{-,\mu} \circ \xi_{+,\mu}^{-1}(y)$, we have $\cb_\mu = \xi_{-,\mu} \circ \xi_{+,\mu}^{-1}$. Since $\cf_\mu$ is the inverse of $\cb_\mu$, the result follows.
\end{proof}

The following technical lemma is used to show Assumptions \ref{assumption:forward-backward-sets} and \ref{assumption:extra-dim-independent} later.

\begin{lemma} \label{lemma:rdom:forward-backward}
  Let $\cf_\mu$ and $\cb_\mu$ be the characteristic maps between inflow and outflow boundary defined in \eqref{eq:inflow-to-outflow} for the characteristics defined in Lemma \ref{lemma:rdom:characteristics}. For centered Euclidean ball $B_h$ of radius $h \le 1/2$, let $\bo[,h] = \xi_{+,\mu}(B_h)$. Then $\bo[,h]$ is independent of $\mu$ and  
  \begin{align*}
    \cf_\mu \circ \cb_{\mu_i}(\bo[,h]) \cap \bo[,h] & \ne \emptyset &
    & \Leftrightarrow &
    (B_h + \mu_i) \cap (B_h + \mu) & \ne \emptyset.
  \end{align*}
  for any $\mu \in \param$.
  
\end{lemma}

\begin{proof}

Since $\xi_{+,\mu} = \rdom_\mu(1,\cdot) = \rdom(1, \cdot)$ is independent of $\mu$, so is $\bo[,h] = \xi_{+,\mu}(B_h)$. For any $\mu$ we have 
\begin{align*}
  \cf_\mu \circ \cb_{\mu_i}(\bo[,h]) \cap \bo[,h] & \ne \emptyset  \\
  \Leftrightarrow \cf_\mu \circ \cb_{\mu_i} \circ \xi_{+,\mu_i}(B_h) \cap \xi_{+,\mu}(B_h) & \ne \emptyset  \\
  \Leftrightarrow (\xi_{+,\mu} \circ \xi_{-,\mu}^{-1}) \circ (\xi_{-,\mu_i} \circ \xi_{+,\mu_i}^{-1}) \circ \xi_{+,\mu_i}(B_h) \cap \xi_{+,\mu}(B_h) & \ne \emptyset  \\
  \Leftrightarrow \xi_{-,\mu}^{-1} \circ \xi_{-,\mu_i} (B_h) \cap B_h & \ne \emptyset,
\end{align*}
where in the first step we have used that $\bo[,h] = \xi_{+,\mu}(B_h) = \xi_{+,\mu_i}(B_h)$ because $\xi_{+,\mu}$ is independent of $\mu$. In the second step we have plugged in the forward/backward characteristics from lemma \ref{lemma:rdom:inflow-to-outflow} and in the last step we have canceled the first $\xi_{+,\mu}$ because it is one-to-one. Since also $\xi_{-, \mu}$ is one-to-one, the last expression is equivalent to
\begin{align*}
  \xi_{-,\mu}^{-1} \circ \xi_{-,\mu_i} (B_h) \cap B_h & \ne \emptyset \\
  \Leftrightarrow \xi_{-,\mu_i} (B_h) \cap \xi_{-,\mu} (B_h) & \ne \emptyset \\
  \Leftrightarrow \rdom(0, B_h + \mu_i) \cap \rdom(0, B_h + \mu) & \ne \emptyset \\
  \Leftrightarrow B_h + \mu_i \cap B_h + \mu & \ne \emptyset, \\
\end{align*}
which shows the result.

\end{proof}

The first main result of this section provides a flow field for Proposition \ref{prop:transport-entropy}.

\begin{lemma} \label{lemma:rdom:assumptions}
  
  Assume there is a map $\rdom$ as defined in Assumption \ref{assumption:rdom:reference-domain}. Let $h\ge 0$. Then, there is a parametric flow field $b_\mu$ with $\|b\|_{C^{\sb}(\Omega \times \param)} \lesssim 1$ for $\param = [-1,1]^{d-1}$ and parameters $\mu_i$, $i=1, \dots, n$ with $n \sim h^{-(d-1)}$ such that Assumptions \ref{assumption:backward-forward}, \ref{assumption:backward-sets}, \ref{assumption:forward-backward-sets} and \ref{assumption:smooth} are satisfied with $\diml = (d-1)$.

\end{lemma}

\begin{proof}

\begin{enumerate}

  \item \emph{Definition of the flow field:} We define $b_\mu$ by Lemma \ref{lemma:rdom:characteristics}. By construction, $b_\mu$ is only defined on the image of $\frac{1}{2} [-1,1] \times \frac{1}{2} [-1,1]^{d-1} \times \frac{1}{2}[-1,1]^{d-1}$ (for $t$, $w$ and $\mu$) of the map $\rdom_\mu$. Since this domain is Lipschitz and $b$ bounded in $C^{\sb}$ by construction and the chain rule, we can extend it to the full space-parameter domain. Nonetheless, in the following, we only use its definition on the subdomain of Lemma \ref{lemma:rdom:characteristics}.

  \item \emph{Definition of $\mu_i$:} We choose $\mu_i$, $i=1, \dots, n$ on a uniform Cartesian grid, contained in $\frac{1}{2}[-1,1]^{d-1}$ with distances $|\mu_i - \mu_j| \ge 5h$ and $n \sim h^{-(d-1)}$ elements.

  \item \emph{Assumption \ref{assumption:backward-forward}:} Follows directly from Lemma \ref{lemma:rdom:inflow-to-outflow}.

  \item \emph{Assumption \ref{assumption:backward-sets}:} For centered Euclidean ball $B_h$ and $\bo[,h] = \xi_{+,\mu}(B_h)$ by Lemma \ref{lemma:rdom:inflow-to-outflow}, we have
  \begin{equation*}
    \cb_{\mu_i}(\bo[,h])
    = \cb_{\mu_i} \circ \xi_{+,\mu} (B_h)
    = \xi_{-,\mu_i} \circ \xi_{+,\mu_i}^{-1} \circ \xi_{+,\mu} (B_h)
    = \xi_{-,\mu_i} (B_h),
  \end{equation*}
  where in the last step we have used that $\xi_{+,\mu} = \rdom_\mu(1,\cdot) = \rdom(1, \cdot)$ is independent of $\mu$. All $\xi_{-,\mu_i}(B_h) = \rdom_{\mu_i}(0, B_h) = \rdom(0,B_h + \mu_i)$ are disjoint because $\rdom$ is one-to-one and all $B_h + \mu_i$ are disjoint by the spacing of the $\mu_i$ grid points.
  
  \item \emph{Assumption \ref{assumption:forward-backward-sets}:} By Lemma \ref{lemma:rdom:forward-backward}, for any $\mu$ we have 
  \begin{align*}
    \cf_\mu \circ \cb_{\mu_i}(\bo[,h]) \cap \bo[,h] & \ne \emptyset &
    & \Leftrightarrow &
    B_h + \mu_i \cap B_h + \mu & \ne \emptyset.
  \end{align*}
  This intersection is non-empty only if $|\mu - \mu_i| \le 2h$. Hence the sets $M_i = \{\mu \in \param : \, \cf_\mu \circ \cb_{\mu_i}(\bo[,h]) \cap \bo[,h] \ne \emptyset \}$ are contained in the balls $B_{2h}(\mu_i)$ and therefore disjoint because the $\mu_i$ are spaced with distance at least $5h$.

  \item \emph{Assumption \ref{assumption:smooth}:} Follows from the smoothness of $\rdom$, Lemma \ref{lemma:rdom:inflow-to-outflow} and the chain rule.

\end{enumerate}

\end{proof}

The second main result of this section provides flow fields for Proposition \ref{prop:transport-entropy-var}.

\begin{lemma} \label{lemma:rdom:assumptions-extension}
  
  Assume there is a map $\rdom$ as defined in Assumption \ref{assumption:rdom:reference-domain}. Let $h\ge 0$. Then, there is a parametric flow field $b_\mu$ with $\|b\|_{C^{\sb}(\Omega \times \param)} \lesssim 1$ for $\param = \parama \times \paramb = [-1,1]^{d-2} \times [-1,1]$ and parameters $\mu_i = (\mua_i, 5h)$, $i=1, \dots, n$ with $n \sim h^{-(d-2)}$ such that Assumptions \ref{assumption:backward-forward}, \ref{assumption:backward-sets}, \ref{assumption:forward-backward-sets}, \ref{assumption:smooth} and \ref{assumption:extra-dim-independent} are satisfied with $\diml = (d-2)$.

\end{lemma}

\begin{proof}

The parameters are split into two components $(\mua, \eta)$, with $\mua \in \real^{d-2}$ and $\eta \in \real$. We fix the last axis and choose parameters $\mu_i = (\mua_i, 5h)$, $i=1, \dots n$ on a uniform Cartesian grid with mesh size $5h$ and $n \sim h^{-(d-2)}$. Then Assumptions \ref{assumption:backward-forward}, \ref{assumption:backward-sets}, \ref{assumption:forward-backward-sets} and \ref{assumption:smooth} are shown identical to Lemma \ref{lemma:rdom:assumptions}, which only relies on the spacing between the parameters, unchanged in our case.

For the remaining Assumption \ref{assumption:extra-dim-independent}, by Lemma \ref{lemma:rdom:forward-backward}, we have
\begin{align*}
  \cf_{\mua,0} \circ \cb_{(\mua_i,5h)}(\bo[,h]) \cap \bo[,h] & \ne \emptyset & 
  & \Leftrightarrow &
  B_h + (\mua_i,5h) \cap B_h + (\mua,0) & \ne \emptyset, \\
\end{align*}
The last parameter component, $0$ versus $5h$, ensures that the intersection is indeed always zero, which shows Assumption \ref{assumption:extra-dim-independent}.
  
\end{proof}

\subsection{Proof of Main Result}
\label{sec:transport:proof-main}

\begin{proof}[Proof of Theorem \ref{th:transport-entropy}]
  The theorem follows directly from Proposition \ref{prop:transport-entropy} with flow field constructed in Lemma \ref{lemma:rdom:assumptions}.
\end{proof}

\begin{proof}[Proof of Theorem \ref{th:transport-entropy-var}]
  The theorem follows directly from Proposition \ref{prop:transport-entropy-var} with flow field constructed in Lemma \ref{lemma:rdom:assumptions-extension}.
\end{proof}

\subsection{Upper Bounds}
\label{sec:transport:upper-bounds}

In this section, we consider some upper bounds for the entropy by constructing an approximation of the quantities of interest $\qi(\mu)$. To this end, recall that in simple cases (Section \ref{sec:transport-simple}) the quantity of interest is a convolution of $g_- \in W^{\sm,p}$ and $g_+ \in W^{\sp,\infty}$. Therefore, we expect smoothness $\sm+\sp$ for $\qi(\mu)$, which directly implies approximation properties. We show that similar observations are true for more general cases as well.

As usual, we consider the parametric PDE
\begin{align*}
  \qi(\mu) & = \int_{\Gamma_+} g_+ u_\mu, &
  b_\mu \nabla u_\mu & = 0, &
  u_\mu|_{\Gamma_{-}} & = g_-
\end{align*}

\begin{lemma} \label{lemma:transport:upper}

  Let $\cf_\mu$ and $\cb_\mu$ be the characteristic maps between inflow and outflow boundary \eqref{eq:inflow-to-outflow} and assume they are well-defined. Assume all derivatives below are continuous and bounded by
  \begin{align*}
    \|g_-\|_{C^{\sm}(\Gamma_-)} & \lesssim 1,
    \\
    \|g_+\|_{C^{\sp}(\Gamma_+)} & \lesssim 1
    \\
    | \partial_x^\delta \partial_\mu^\eta \cf_\mu(x) | & \lesssim 1, &
    |\delta| & \le 1, & 
    |\eta| & \le \sp,
    \\
    | \partial_x^\delta \partial_\mu^\eta \cb_\mu(y) | & \lesssim 1, &
    |\delta| & \le \sm, & 
    |\eta| & \le \sp,
  \end{align*}
  for all $x \in \Gamma_-$, $y \in \Gamma_+$ and $\mu \in \param$. Furthermore, assume that $\det D \cf_\mu$ is uniformly positive or negative. Then
  \begin{equation*}
    \|\qi\|_{C^{\sm+\sp}(\param)} \lesssim 1.
  \end{equation*}
  
\end{lemma}

\begin{proof}

By the method of characteristics, the solution of the PDE is given by $g_- \circ \cb_\mu$. For multi-index $\alpha$ with $|\alpha| \le \sm+\sp$, we denote the partial derivatives with respect to $\mu$ by $\partial^\alpha$. We split the multi-index into two arbitrary components $\beta$ and $\gamma$ so that $\alpha = \beta + \gamma$ and $|\beta| \le \sm$ and $|\gamma| \le \sp$. Since $g_+$ is independent of $\mu$, we have
\begin{align*}
  \mathcal{D} & := \partial^\alpha \qi(\mu) = \partial^\alpha \int_{\Gamma_+} g_+ (g_- \circ \cb_\mu)
  = \partial^\gamma \int_{\Gamma_+} g_+  \partial^\beta( g_- \circ \cb_\mu )
  \\
  & = \partial^\gamma \int_{\Gamma_+} g_+  \sum_{\pi \in \Pi} (g_-^{(|\pi|)} \circ \cb_\mu) \prod_{\delta \in \pi} \partial^\delta \cb_\mu,
\end{align*}
where in the last line we have used Fa\`{a} di Bruno's formula, i.e. a repeated chain rule for higher derivatives, and the sum runs through all partitions $\Pi$ of $\{0, \dots, |\beta|\}$ and $\delta \in \pi$ runs through all subsets in the partition $\pi$. We place the remaining $\partial^\gamma$ derivatives on $g_+$ instead of $g_-$. To this end, we first apply the substitution rule
\begin{align*}
  \mathcal{D}
  & = \partial^\gamma \int_{\Gamma_+} g_+ \circ \cf_\mu  \sum_{\pi \in \Pi} g_-^{(|\pi|)} \prod_{\delta \in \pi} (\partial^\delta \cb_\mu) \circ \cf_\mu |\det D \cf_\mu| 
  \\
  & = \sum_{\pi \in \Pi} \int_{\Gamma_+} g_-^{(|\pi|)} \partial^\gamma \left(g_+ \circ \cf_\mu  \prod_{\delta \in \pi} (\partial^\delta \cb_\mu) \circ \cf_\mu |\det D \cf_\mu| \right),
\end{align*}
where in the last line we have used that after substitution $g_-^{(|\pi|)}$ does not depend on $\mu$ and $\cf_\mu = \cb_\mu^{-1}$. By product and chain rules, expanding the $\partial^\gamma$ derivative of the parenthesis yields terms
\begin{align*}
  & g_+^{(|\eta|)}, &
  & \partial^\eta \left[ (\partial^\delta \cb_\mu) \circ \cf_\mu \right]
  & \partial^\eta (\det D \cf_\mu)
\end{align*}
for any $\eta \le \gamma$ (component-wise). For the last term we omitted the absolute value because by assumption it does not change sign. Note that the middle term requires $x$ and $\mu$ derivatives of $\cb_\mu$, whereas we only need $\mu$ derivatives of $\cf_\mu$ and $D \cf_\mu = D_x \cf_\mu$. By assumption, all these derivatives are well-defined and bounded in the maximum norm, which directly yields the lemma.

\end{proof}

\begin{corollary} \label{cor:transport-upper}

  For parameter set $\param \subset \real^D$, let $\instances$ be a class of transport problems that satisfy the assumptions of Lemma \ref{lemma:transport:upper}. Then the Lipschitz width is bounded by
  \begin{equation*}
    \mwidth^\gamma_n(\instances) \le n^{-\frac{\sm+\sp}{D}}.
  \end{equation*}
  
\end{corollary}

\begin{proof}
  
Follows from the smoothness in Lemma \ref{lemma:transport:upper} and approximation of the quantity of interest by a standard method, e.g. spline, finite element or wavelet.
  
\end{proof}

\begin{remark}
  Since the lower bounds in Theorem \ref{th:transport-entropy-var} or Corollary \ref{cor:transport-entropy-var} and the upper bounds in Corollary \ref{cor:transport-upper} do not match, the results are not sharp. Nonetheless, both bounds share relevant qualitatively properties: They depend on the combined smoothness $\sm+\sp$ of the inflow boundary condition and outflow quantity of interest and are parameter dimension dependent.
\end{remark}

\section{Implications for Transform Based Methods}
\label{sec:implications-for-transform-methods}

Classical reduced order models, like the reduced basis method or POD, represent solutions by separation of variables, or tensor product approximation,
\begin{equation*}
  u_\mu(x) = \sum_{i=1}^n c_i(\mu) \psi_i(x),
\end{equation*}
with reduced basis functions $\psi_i(x)$ and parameter dependent coefficients $c_i(\mu)$.  Recent approaches 
\cite{
  Welper2017,
  ReissSchulzeSesterhenn2015,
  CagniartMadayStamm2019,
  CagniartCrisovanMadayEtAl2017,
  NairBalajewicz2017,
  Welper2019,
  KrahSrokaReiss2019,
  Taddei2020%
}
for parametric transport dominated problems add an inner parameter dependent transform $\varphi_\mu$
\begin{equation*}
  u_\mu(x) = \sum_{i=1}^n c_i(\mu) \psi_i \circ \varphi_\mu(x)
\end{equation*}
to track the transport of sharp gradients or shocks. While these can achieve excellent approximation rates for the parametric solution $u_\mu$, a full offline/online decomposition is not clear, in general. Indeed, if we apply a linear functional, or quantity of interest, $\ell(\cdot) = \int g \cdot$ in the classical case we obtain
\begin{equation*}
  \ell(u_\mu) 
  = \sum_{i=1}^n c_i(\mu) \ell(\psi_i) 
  =: \sum_{i=1}^n c_i(\mu) \int \psi_i \, g
\end{equation*}
and we only need to retain the parameter independent numbers $\ell(\psi_i)$ for the online phase. With the additional inner transform, however, these quantities are parameter dependent
\begin{equation*}
  \ell(u_\mu) 
  =: \sum_{i=1}^n c_i(\mu) \int (\psi_i \circ \varphi_\mu) \, g,
\end{equation*}
or after integral substitution
\begin{equation*}
  \ell(u_\mu) = \sum_{i=1}^n c_i(\mu) \int \psi_i \, \left(g \circ \varphi_\mu^{-1}\right) \, | \det D \varphi_\mu^{-1}|.
\end{equation*}
This leaves us with the task to find a reduced approximation of either $\psi_i \circ \varphi_\mu$ or $\left(g \circ \varphi_\mu^{-1}\right) \, | \det D \varphi_\mu^{-1}|$. Since $\psi_i$ is typically composed of snapshots, the former variant requires the PDE solutions to be smooth, which for the transport problems in Section \ref{sec:transport} translates to smooth inflow boundary conditions. The latter variant requires some regularity of the quantity of interest $g$.

Corollary \ref{cor:transport-entropy-var} suggests that this issue is not only of technical nature but a fundamental one: Any stable high order reduced order model with full online/offline decomposition (applied to the transport problems in this paper) requires either smooth inflow $g_-$ and therefore smooth solutions, or smooth quantity of interest $g = g_+$.

\section{Conclusion}

In order to find benchmarks for nonlinear reduced order models, it is tempting to generalize the Kolmogorov $n$-width of the solution manifold to nonlinear width. We have seen that this approach has several difficulties: It is unclear if optimal encoder/decoder pairs allow an online/offline decomposition and the widths tend to be zero for more degrees of freedom than parameter dimension.

These problems are avoided by a shift in perspective: Instead of applying the nonlinear width to the solution manifold, the offline and online phases themselves can be understood as a encoder/decoder pair to which we can apply manifold, stable and Lipschitz width. 

With the new approach, we show that reduced order models for linear transport, with differentiable inflow boundary condition, flow field and quantity of interest, are necessarily subject to the curse of dimensionality and achieve high order only if either the inflow boundary condition or outflow quantity of interest are smooth.

\section{Acknowledgements}

We are grateful for helpful discussions on early results for the paper with Peter Binev, Andrea Bonito, Ronald DeVore, Guergana Petrova and Bojan Popov.

\bibliographystyle{abbrv}
\bibliography{entropy}

\end{document}